\documentclass[reqno]{amsart}
\usepackage{textcomp}

\usepackage[dvipsnames]{xcolor}
\usepackage{amsmath}
\usepackage{amsthm}
\usepackage{hyperref}
\usepackage{amsfonts,graphics,amsthm,amsfonts,amscd,latexsym}
\usepackage{tikz-cd}
\usepackage{epsfig}
\usepackage{flafter}
\usepackage{longtable}
\usepackage{mathtools}
\usepackage{comment}
\usepackage{stmaryrd}

\usepackage{mathabx,epsfig}

\hypersetup{
    colorlinks=true,    
    linkcolor=blue,          
    citecolor=blue,      
    filecolor=blue,      
    urlcolor=blue           
}
\usepackage{tikz}
\usetikzlibrary{graphs,positioning,arrows,shapes.misc,decorations.pathmorphing}

\tikzset{
    >=stealth,
    every picture/.style={thick},
    graphs/every graph/.style={empty nodes},
}

\tikzstyle{vertex}=[
    draw,
    circle,
    fill=black,
    inner sep=1pt,
    minimum width=5pt,
]
\usepackage[position=top]{subfig}
\usepackage{amssymb}
\usepackage{color}

\calclayout

\usetikzlibrary{decorations.pathmorphing}
\tikzstyle{printersafe}=[decoration={snake,amplitude=0pt}]

\newcommand{\Cl}{\operatorname{Cl}}
\newcommand{\Pic}{\operatorname{Pic}}

\newcommand{\supp}{\operatorname{supp}}

\newcommand{\Spec}{\operatorname{Spec}}

\newcommand{\ddivv}{\operatorname{div}}

\newcommand{\Cox}{\operatorname{Cox}}

\newcommand{\PSL}{\operatorname{PSL}}
\newcommand{\SL}{\operatorname{SL}}
\newcommand{\Li}{\mathcal{L}}

\newcommand{\pp}{\mathbb{P}}

\newcommand{\qq}{\mathbb{Q}}
\newcommand{\zz}{\mathbb{Z}}
\newcommand{\nn}{\mathbb{N}}
\newcommand{\rr}{\mathbb{R}}

\newcommand{\ff}{\mathbb{F}}
\newcommand{\kk}{\mathbb{K}}

\newcommand{\oo}{\mathcal{O}}

\def\O#1.{\mathcal {O}_{#1}}   
\def\pr #1.{\mathbb P^{#1}}    
\def\af #1.{\mathbb A^{#1}}   
\def\ses#1.#2.#3.{0\to #1\to #2\to #3 \to 0} 
\def\xrar#1.{\xrightarrow{#1}}   
\def\K#1.{K_{#1}}      
\def\bA#1.{\mathbf{A}_{#1}}   
\def\bM#1.{\mathbf{M}_{#1}}    
\def\bL#1.{\mathbf{L}_{#1}}    
\def\bB#1.{\mathbf{B}_{#1}}    
\def\bK#1.{\mathbf{K}_{#1}}   
\def\subs#1.{_{#1}}     
\def\sups#1.{^{#1}}

\newcommand{\lif}{\upharpoonleft \kern-0.35em}
\newcommand{\ouclu}{\overline{\mathcal{U}}}
\newcommand{\uclu}{\mathcal{U}}
\newcommand{\clu}{\mathcal{A}}
\newcommand{\val}{\mathcal{V}}

\usepackage{tikz}
\usetikzlibrary{matrix,arrows,decorations.pathmorphing}

\newtheorem{introthm}{Theorem}[section]
\newtheorem{introcor}[introthm]{Corollary}

  \newtheorem{theorem}{Theorem}[section]
  \newtheorem{lemma}[theorem]{Lemma}
  \newtheorem{proposition}[theorem]{Proposition}
  \newtheorem{corollary}[theorem]{Corollary}

  \newtheorem{definition}[theorem]{Definition}
  \newtheorem{example}[theorem]{Example}

\newtheorem{remark}[theorem]{Remark}

\theoremstyle{remark}

\numberwithin{equation}{section}

\usepackage[all]{xy}

\begin{document}

\thanks{J.M. was partially supported by NSF research grant DMS-2443425.}

\title[Fano compactifications of mutation algebras]
{Fano compactifications of mutation algebras}

\author[J.~Enwright]{Joshua Enwright}
\address{UCLA Mathematics Department, Box 951555, Los Angeles, CA 90095-1555, USA
}
\email{jlenwright1@math.ucla.edu }

\author[L.~Francone]{Luca Francone}
\address{Università di Roma Tor Vergata, Roma, ITA
}
\email{francone@mat.uniroma2.it}

\author[J.~Moraga]{Joaqu\'in Moraga}
\address{UCLA Mathematics Department, Box 951555, Los Angeles, CA 90095-1555, USA
}
\email{jmoraga@math.ucla.edu}

\author[H.~Spink]{Hunter Spink}
\address{Department of Mathematics, University of Toronto, Toronto.}
\email{hunter.spink@utoronto.ca}


\begin{abstract}
In this article, 
we introduce the notion of mutation semigroup algebras. 
This concept simultaneously generalizes 
cluster algebras and semigroup algebras.
We show that, under some mild conditions on the singularities,
the spectrum $U={\rm Spec}(R)$ of a mutation semigroup algebra $R$ admits a log Fano compactification $U\hookrightarrow X$. 
The compactification $X$ can be chosen to be a $\qq$-factorial log Fano variety whenever $U$ is $\qq$-factorial.
Furthermore, we prove that a $\qq$-factorial klt Fano variety $X$ is of cluster type if and only if its Cox ring ${\rm Cox}(X)$ is a ${\rm Cl}(X)$-graded mutation semigroup algebra.
In order to enlighten the previous theorems, 
we provide several explicit examples motivated by birational geometry, representation theory, and combinatorics.
\end{abstract} 

\maketitle

\setcounter{tocdepth}{1} 
\tableofcontents

\section{Introduction}
Cluster and upper cluster algebras are commutative rings of combinatorial nature defined by Fomin and Zelevinsky~\cite{FZ02,FZ03,FZ05,FZ07}. In the last two decades, cluster algebras have had a tremendous impact on different fields of mathematics, including combinatorics~\cite{chapoton02polytopal, fomin05generalized, ingalls09noncrossing, RS16}, Lie theory \cite{Lec10, geiss08partial, geiss2011kac,  Fei21tensor, Francone2023minimal,GLSB23, CGGLSS25} representation theory of quantum algebras \cite{hernandez10cluster,hernandez15cluster, hernandez16cluster, qin17triangular, geiss24representations, kashiwara20monoidal,kashiwara24monoidal}, Teichm\"uler theory~\cite{FG06,fock09quantum,fock09cluster} and Poisson geometry \cite{berenstein05quantum,gekhtman05cluster,gekhtman10cluster}. 
Cluster algebras are also important in Mirror Symmetry due to the work of Gross, Hacking, Keel, and Kontsevich~\cite{GHK15,GHKK18}. Cluster algebras are defined in terms of seeds and mutations between seeds. The mutations are binomial transformations of elements across different seeds (see Definition~\ref{def:CA}).
In~\cite{LP16}, Lam and Pylyavskyy introduced the concept of Laurent phenomenon algebras (LPAs), which generalize the concept of cluster algebras, by allowing more general mutations associated to irreducible Laurent polynomials. 
They proved that many good properties of cluster algebras, such as the Laurent phenomenon, are also valid for LPAs. 
In this article, we introduce two new notions of commutative rings. Mutations algebras and mutation semigroup algebras. 
The former, already suggested by~\cite[Definition 6]{corti2023cluster}, is a more general form of commutative algebra defined by an initial seed and mutations associated with possibly reducible polynomials. 
The definition of mutation semigroup algebra is purely algebraic (see Definition~\ref{def:MSA}). Essentially, a mutation semigroup algebra 
is the intersection of a mutation algebra with a semigroup algebra, 
in their common function field.
Two of the aims of this article are to show that mutation semigroup algebras are deeply connected to the geometry of Fano varieties and Cox rings of algebraic varieties. 
In Subsection~\ref{subsec:FC}, we will explain results regarding Fano compactifications of the spectra of mutation algebras. 
In Subsection~\ref{subsec:CR}, we will investigate the geometry of a variety $X$ when its Cox ring admits the structure of a mutation semigroup algebra. 
Finally, in Subsection~\ref{subsec:comb}, we will provide several explicit examples motivated by Lie theory and their combinatorics.

\subsection{Fano compactifications}\label{subsec:FC}
A {\em cluster type variety} is a partial compactification of an algebraic torus 
on which the standard volume form of the algebraic torus has no zeros (see Definition~\ref{def:ct}).
The divisor given by the poles of the volume form is known as a {\em cluster type boundary}. 
Cluster type varieties provide a generalization of cluster varieties. This notion was introduced by Enwright, Figueroa, and Moraga in~\cite[Definition 2.25]{EFM24}; however, Corti had already given a very similar definition in~\cite[Definition 2]{corti2023cluster}. 
Our first result characterizes mutation semigroup algebras as rings of regular functions 
on the complement of an ample divisor supported on a cluster type boundary.

\begin{introthm}\label{introthm:msa-defs-agree}
Let $R$ be a finitely generated commutative $\kk$-algebra. 
Assume that ${\rm Spec}(R)$ has klt singularities. 
Then, the following statements are equivalent:
\begin{enumerate}
\item $R$ is a mutation semigroup algebra, 
\item there exists a normal projective variety $X$, a cluster type pair $(X,B)$, and an ample divisor $0\leq A\leq B$ such that the isomorphism 
$R \simeq H^0(X\setminus A, \mathcal{O}_X)$ holds.
\end{enumerate} 
\end{introthm}

We refer the reader to Definition~\ref{def:klt} for the notion of klt singularities.
Compactifications of the spectra of cluster algebras have been studied by Gross, Hacking, Keel, and Kontsevich~\cite{GHKK18}. In such an article, these compactifications are called {\em partial minimal models}. In~\cite{CMNC22}, Cheung, Magee, and N\'ajera Ch\'avez use a notion of convexity in the ambient space of the scattering diagram of a cluster variety to find a suitable projectivization of its dual. 
Conversely, in~\cite{Francone2024Cox} Francone has proved that, under  mild conditions, the Cox ring ${\rm Cox}(X)$ of a
compactification $X$ of the spectrum of an upper cluster algebra admits a cluster algebra structure.
The proof of Francone uses the monomial lifting method (see~\cite[Section 2.2]{Francone2024Cox}).
Relations between Cox rings and cluster algebras are also discussed in \cite{GHK15birational, mandel19cluster,cheung25valutative}.
Theorem~\ref{introthm:msa-defs-agree} can be thought of as giving a partial minimal model for any mutation semigroup algebra whose spectrum has klt singularities. 
Conversely, cluster type varieties can be viewed as 
partial minimal models of the spectra of mutation semigroup algebras.
We expect the condition on the singularities to be very mild. Indeed, the authors know no explicit mutation semigroup algebra whose spectrum has worse than klt singularities.
Our next theorem asserts that the projective models mentioned in Theorem~\ref{introthm:msa-defs-agree} can be chosen to be Fano or near Fano.

\begin{introcor}\label{introcor:msa-Fano-comp}
Let $R$ be a finitely generated commutative $\kk$-algebra.
Assume that $R$ is a mutation semigroup algebra 
and that $U={\rm Spec}(R)$ has klt singularities. 
Then, the following statements hold: 
\begin{enumerate}
\item the affine variety $U$ admits a log Fano compactification, and 
\item if the volume form $\Omega_U$ of $U$ has no poles on $U$, then $U$ admits a klt Fano compactification.
\end{enumerate} 
Furthermore, if $U$ is $\qq$-factorial, then the two previous compactifications can be chosen to be $\qq$-factorial.
\end{introcor} 

The volume form $\Omega_U$ on the spectrum $U$ of a mutation semigroup algebra $R$ is introduced in Definition~\ref{def:MSA}. 
In the case of cluster algebras, this agrees with the volume form induced by one of the algebraic tori embedded in its spectrum.
The proof of both Theorem~\ref{introthm:msa-defs-agree} and Corollary~\ref{introcor:msa-Fano-comp} use some technical tools from birational geometry, the minimal model program, and Fano geometry. 
In view of Corollary~\ref{introcor:msa-Fano-comp}, it would be interesting to study explicit klt Fano compactifications of cluster algebras coming from combinatorics. 
On the other hand, one can seek optimal Fano compactifications of cluster algebras. 
For instance, many invariants of Fano varieties, such as the volume, singularities, etc, could lead to some special Fano compactification of the spectra of cluster algebras. In Section~\ref{sec:ex-ex}, we will show several explicit examples of Fano compactifications of cluster algebras. 
Furthermore, we show that compactifications that add only poles of $\Omega_U$ must be log Fano.

\begin{introcor}\label{introcor:nice-comp-FT}
Let $R$ be a mutation semigroup algebra 
and $U={\rm Spec}(R)$. 
Let $\Omega_U$ be the volume form of $U$.
Assume that $U$ is klt. 
Let $U\hookrightarrow X$ be a compactification for which
$\Omega_U$ has a pole at every prime divisor in $X\setminus U$. 
Then, the variety $X$ is log Fano.
In particular, it is a Mori dream space.
\end{introcor} 

For $R=\kk[x_1^{\pm},\dots,x_n^{\pm}]$ the varieies $X$ appearing in Corollary~\ref{introcor:nice-comp-FT} are precisely complete toric varieties (see, e.g,~\cite[Theorem 1.10]{EFM24}).
The spectra of locally acyclic cluster algebras 
are known to  be finitely generated by \cite{Muller13acyclic} and to have canonical singularities
due to the work of Benito, Muller, Rajchgot, and Smith~\cite{BMRS15}.
Thus, from Corollary~\ref{introcor:msa-Fano-comp}.(2), we obtain the following statement about klt Fano compactifications of locally acyclic cluster algebras. 
\vspace{0.3cm}

\begin{introcor}~\label{introcor:comp-locally-acyclic}
Let $R$ be a locally acyclic cluster algebra and $U={\rm Spec}(R)$.
Then, the following statements hold:
\begin{enumerate}
\item $U$ admits a klt Fano compactification, and 
\item $U$ admits a canonical log Fano compactification.
\end{enumerate} 
\end{introcor}

Let us emphasize that the two compactifications may differ. 
Indeed, the compactification considered in Corollary~\ref{introcor:comp-locally-acyclic}.(1) may not have canonical singularities. 
Indeed, canonical Fano varieties of a fixed dimension are bounded~\cite{Bir21}. Therefore, $n$-dimensional locally acyclic cluster algebras admitting canonical Fano compactifications must have an upper bound for the rank of their class groups. In Subsection~\ref{ex:affinisation-rank-2-cluster}, we make explicit the difference between both compactifications.

\subsection{Cox rings}\label{subsec:CR}
The Cox ring of an algebraic variety is a generalization 
of the concept of homogeneous coordinate rings for toric varieties. See Definition~\ref{def:CR} for the notion of Cox rings.
A variety whose Cox ring is finitely generated is known as a Mori dream space~\cite{HK00}. 
Mildly singular Fano varieties are known to be Mori dream spaces~\cite{BCHM10}. 
Theorem~\ref{introthm:msa-defs-agree} and Corollary~\ref{introcor:msa-Fano-comp} imply that cluster type Fano varieties are Fano compactifications of the spectra of mutation semigroup algebras.
Following the philosophy of Francone, initiated in~\cite{Francone2023minimal,Francone2024Cox}, we expect that the Cox ring of a Fano cluster type variety has a structure similar to that of a mutation semigroup algebra.
Our next theorem states that for a klt Fano variety $X$ the cluster type condition of $X$ is equivalent to a graded version of the mutation semigroup algebra condition on its Cox ring ${\rm Cox}(X)$. 

\begin{introthm}\label{introthm:cluster-type-Cox-rings}
Let $X$ be a $\qq$-factorial klt Fano variety. Then $X$ is of cluster type if and only if
${\rm Cox}(X)$ is a ${\rm Cl}(X)$-graded mutation semigroup algebra. 
\end{introthm} 

It is a classic fact that cluster varieties are covered by 
algebraic tori up to a closed subset of codimension two.
In particular, the Cox space of a cluster variety is also covered by algebraic tori up to a closed subset of codimension two. 
It is natural to wonder whether this property also holds for cluster type Fano varieties. 
For cluster type Fano varieties, this condition is known to hold in the complement of the poles of the volume form~\cite[Theorem 1.3]{EFM24}.
In this direction, we prove that the Cox space of a cluster type Fano variety is covered by algebraic tori up to a closed subset of codimension two.

\begin{introthm}\label{introthm:Cox-space-covered-algebraic-torus}
Let $X$ be a $\qq$-factorial klt Fano variety of dimension $n$
and Picard rank $\rho$. 
If $X$ is of cluster type, then its Cox space $\overline{X}$ 
is covered by copies of $\mathbb{G}_m^{n+\rho}$ up
to a closed subset of codimension at least two. 
\end{introthm}

The proof of Theorem~\ref{introthm:Cox-space-covered-algebraic-torus} shows that any cluster type surface is indeed covered by copies of $\mathbb{G}_m^2$ up to finitely many points.

\subsection{Projective varieties from Lie theory}\label{subsec:comb}
In recent years, there has been considerable interest in finding cluster algebra structures on the coordinate rings and Cox rings of varieties arising from Lie theory, such as Grassmannians, flag varieties, Schubert varieties, Richardson varieties, and braid varieties.
The literature in this direction is extensive (see, e.g.,~\cite{Sco06, geiss08partial, geiss2011kac, SSBW19,MR20,GL23, GLSB23, CGGLSS25}).
In this subsection, we explain that many projective varieties from Lie theory are indeed cluster type varieties.
The following theorem is a recollection from results due to Knutson, Lam, and Speyer~\cite{KLS14}, Brion~\cite{Bri02}, and Kumar~\cite{BK05}.

\begin{introthm}\label{introthm:brick-richardson}
Let $\mathring{X}^w_v$ be an open Richardson variety.
Then, there exists a brick manifold
compactification $\mathring{X}^w_v\hookrightarrow \widetilde{X}^w_v$ 
satisfying the following conditions:
\begin{enumerate}
\item the variety $\widetilde{X}^w_v$ is smooth; 
\item the divisor $\widetilde{B}^w_v:=\widetilde{X}^w_v\setminus \mathring{X}^w_v$ is anti-canonical; 
\item the divisor $\widetilde{B}^w_v$ supports an effective ample divisor; and 
\item the divisor $\widetilde{B}^w_v$ is simple normal crossing.
\end{enumerate} 
In particular, the brick  manifold compactification of an open Richardson variety is a log Fano cluster type variety.
\end{introthm} 

Open Richardson varieties are known to be spectra of locally acyclic cluster algebras (see~\cite{MR20}). Thus, the previous theorem can be thought of as a special case of Corollary~\ref{introcor:comp-locally-acyclic}.(2). 
As a consequence of Theorem~\ref{introthm:brick-richardson}, we will deduce that many examples of varieties from Lie theory
are log Fano and cluster type.

\begin{introcor}\label{introcor:var-from-comb-ct}
The following classes of varieties: 
\begin{enumerate}
\item Brick manifold compactifications of open Richardson varieties;
\item Richardson varieties;
\item Flag varieties;
\item Schubert varieties; and 
\item Bott-Samelson varieties;
\end{enumerate} 
are log Fano and cluster type.
\end{introcor}

It is worth mentioning that the fact that Bott-Samelson varieties and Schubert varieties are log Fano was proved by Anderson and Stapledon~\cite{AS14}. Furthermore, the fact that Richardson varieties are log Fano was proved by Schwede and Kumar~\cite{KS14}. 

The article is organized as follows. 
In Section~\ref{sec:prel}, we will recall all the definitions needed to understand the main theorems of the article and the examples in Section~\ref{sec:ex-ex}. These two sections are aimed at a more general audience, and we avoid technical tools as much as possible.
In Section~\ref{sec:ex-ex}, we give several explicit examples 
of Theorem~\ref{introthm:msa-defs-agree}
and Theorem~\ref{introthm:cluster-type-Cox-rings}
related to Schubert varieties (Subsection~\ref{subsec:schubert}), Richardson varieties (Subsection~\ref{subsec:richardson}), resolutions
of cubic surfaces (Subsection~\ref{subsec:cubic}), and blow-ups of the projective space (Subsection~\ref{subsec:blow-up-pn}), among others. 
In Section~\ref{sec:more-prel}, we give more preliminary results regarding the geometry of Fano varieties, the minimal model program, and birational geometry. 
In Section~\ref{sec:comb}, based on previous works, we prove that many varieties from Lie theory are of cluster type.
In Section~\ref{sec:geom}, we prove the results regarding Fano compactifications of the spectra of mutation algebras.
In Section~\ref{sec:MA-vs-CR}, we prove the results regarding mutation algebra structures on the Cox rings of algebraic varieties. 

\section{Preliminaries}\label{sec:prel}

We work over an algebraically closed field $\kk$ of characteristic zero. 
In this section, we introduce some preliminary definitions regarding mutation algebras, singularities, and Fano varieties.

\subsection{Mutation semigroup algebras}

In this subsection, we introduce the definition of mutation semigroup algebras. To do so, we first recall the concepts of cluster algebras and Laurent phenomenon algebras.

\subsubsection{Cluster algebras}

We briefly recall the definition and some key properties of cluster and upper cluster algebras of geometric type, following \cite{FZ02} and \cite{FZ05}.
Let $I= I_{uf} \sqcup I_{f}$  be a finite set with a fixed partition and $\mathbb{F}$ be a field extension of $\kk$.
Let $B \in \zz^{I \times I_{uf}}$ be an integer matrix such that its \textit{principal part} $B^\circ:= B_{| I_{uf} \times I_{uf}}$ is skew-symmetrizable.
Moreover, let $x:=(x_i)_{i \in I}$ be a collection of elements of $\mathbb{F}^\times$ that freely generate the field $\mathbb{F}$ over $\kk.$
The pair $t:=(B,x)$ is a \textit{seed} of the field $\mathbb{F}.$
We refer to the set $I$ (resp. $I_{uf}$, resp. $I_{f}$) as the set of \textit{vertices} (resp.  \textit{mutable vertices}, resp.  \textit{frozen vertices}) of the seed $t$.
The matrix $B$ is the \textit{exchange matrix} of $t$ and $x$ is its \textit{cluster}. 
An element $x_i$ for some $i \in I$ is a mutable or frozen \textit{cluster variable} of the seed $t$ according to the type of its vertex.
We denote by 
\begin{equation}
\label{eq:def laurent pol}
\Li(t):= \kk[x_i^{\pm 1} \mid i \in I]
\end{equation}
the ring of Laurent polynomials in the cluster variables of $t$.

We say that two seeds $t=(B,x)$ and $t'=(B',x')$ of the field $\mathbb{F}$ are \textit{mutation equivalent}, written $t' \sim t$, if there exists a finite sequence $k_1, \dots , k_n$ of mutable vertices of $t$ such that 
\[
t'= \mu_{k_n} \circ \cdots \circ \mu_{k_1} (t).
\]
Here, for a mutable vertex $k$, we denote by $\mu_k$ the \textit{mutation in direction $k$} as defined in \cite{FZ02}.
We recall that if $t'= \mu_k(t) \, \, (k \in I_{uf})$, the cluster $x'=(x'_i)_{i \in I}= \mu_k(x)$ of $t'$ is related to the cluster $x$ of $t$ by the fact that $x'_i=x_i$ for $i \neq k$ and that 
\begin{equation}
\label{eq:exchange rel}
x'_k x_k:= \prod_{j \in I} x_j^{\max\{0, b_{jk}\}} +\prod_{j \in I} x_j^{\max\{0, -b_{jk}\}}.
\end{equation}
Formula \eqref{eq:exchange rel} is called  the \textit{exchange relation} of the seed $t$ in direction $k$.
\begin{definition}\label{def:CA}
{\em 
The intersection 
\[
\uclu(t):= \bigcap_{t' \sim t} \Li(t')
\]
of the Laurent polynomial rings in the cluster variables of the seeds mutation equivalent to $t$ is the
{\em upper cluster algebra} of the seed $t.$ 
}
\end{definition} 
We denote by $\clu(t)$ the \textit{cluster algebra} of $t$, which is by definition the subalgebra of $\ff$ generated (over $\kk$) by the cluster variables of the seeds $t' \sim t$ and by the inverse of the frozen variables.
By the Laurent phenomenon \cite{FZ02}, we have an inclusion
\[
\clu(t) \subseteq \uclu(t).
\]
We recall that the seed $t$ is said to be of \textit{maximal rank} (resp. \textit{primitive}) if its exchange matrix $B$ has such a feature.~\footnote{The matrix $B$ is primitive if any of its columns is a primitive vector of $\zz^I$.} 
It is well known that the property of being of maximal rank (resp. primitive) is invariant under seed mutation. 
We recall the following result \cite[Corollary 1.9]{FZ05}.

\begin{theorem}
\label{thm:upper finite}
If the seed $t$ is of maximal rank, then the following equality holds:
\[
\uclu(t)= \Li(t) \cap \bigcap_{ k \in I_{uf}} \Li(\mu_k(t)).
\] 
\end{theorem}
Most of the cluster structures encountered in nature are associated to maximal rank seeds. 
Hence, even though the set of seeds mutation equivalent to a given seed is typically infinite \cite{FZ03}, the intersection defining  upper cluster algebras  can often be assumed to be finite. 

\medskip

Assume from now on that the seed $t$ is fixed, and simply denote by $\uclu$ the upper cluster algebra $\mathcal{U}(t).$
Recall that algebra $\uclu$ can be interpreted as the ring of regular functions on the associated $\clu$-type cluster variety $U$ \cite{Fock2009cluster, GHK15birational}. 
Denote by $U^{\rm{trop}}(\zz)$ the set of tropical points of $U$, as defined in \cite[Definition 1.7]{GHK15birational}. 
We recall that a point $v$ of $U^{\rm{trop}}(\zz)$ is a divisorial discrete valuation $v : \kk(U)= {\rm{Frac}}(\uclu) \longrightarrow \zz \cup \{\infty\}.$
Given a collection ${\bf{v}}= (v_j)_{1 \leq j \leq n} \, \, (n \in \nn)$ of tropical points of $U$, one can consider, as it is done in \cite{GHKK18}, the algebra 
\[
\uclu^{\bf{v}}:= \{ u \in \mathcal{U} \mid v_j(u) \geq 0 \quad (1 \leq j \leq n) \},
\]
which should be interpreted as the algebra of regular functions on a partial compactification of $U$ depending on ${\bf{v}}.$
Notice that for any point $v \in U^{\rm{trop}}(\zz)$ and for any seed $t' \sim t$, there exists an element $v_{t'} \in \zz^I$ such that the discrete valuation $v$, in the coordinates given by the cluster $x'$ of $t'$, reads as
\begin{equation}
\label{eq:tropical vals}
v\bigl( \sum_{m \in \zz^I} a_m (x')^m \bigr)= \min_{m \, : \, a_m \neq 0}  \langle v_{t'}, m \rangle \qquad (a_m \in \kk),
\end{equation}
where the notation $\langle \cdot, \cdot \rangle$ stands for the standard scalar product of $\zz^I.$
It follows that the algebra $\uclu^{\bf v}$ can be expressed as an intersection
\[
\uclu^{\bf v} = \bigcap_{t' \sim t} \Li^{\bf{v}}(t'),
\]
in which each $\Li^{\bf{v}}(t')$ is a semigroup algebra. 
This discussion partially motivates our Definition \ref{def:MSA} of mutation semigroup algebras.
We end this section by emphasizing a relevant and ubiquitous case of this construction, which has found various representation theoretic and geometric applications.
Namely, the algebra 
\begin{equation}
\label{eq:upper clust non-inv}
\ouclu:= \bigcap_{t' \sim t}   \kk[x'_j \mid j \in I_{f}][(x_k')^{\pm 1} \mid k \in I_{uf}]
\end{equation}
is the \textit{upper cluster algebra with non-invertible frozen variables} of the seed $t.$ 
Moreover, we denote by $\overline{\clu}$ the corresponding \textit{cluster algebra with non-invertible frozen variables}: the subalgebra of $\ouclu$ generated by the cluster variables of the seeds mutation equivalent to $t$.
Notice that $\ouclu$ can be interpreted as the algebra $\uclu^{\bf{v}_f}$ for the collection of tropical points ${\bf{v}_f}=(v_i)_{i \in I_f}$ corresponding to the frozen vertices of $\uclu.$
Namely, these are the valuations defined by the fact that (in the notation of Equation \eqref{eq:tropical vals}) $v_{i,t'} \, \, (i \in I_{f}, \, t' \sim t)$ is the $i$-th basis element of $\zz^I$.

\subsubsection{Laurent phenomenon algebras}
Let us briefly recall that \textit{Laurent phenomenon algebras} (LPAs), defined by Lam and Pylyavskyy \cite{LP16}, provide a partial generalization of cluster algebras of geometric type (see \cite[Proposition 4.2]{LP16}). 
As for cluster algebras, LPAs are defined in terms of algebraic objects called \textit{LPA-seeds}, which are all related to each other by certain operations called mutations.
The cluster variables of two adjacent LPA-seeds of a given LPA are related to each other by a generalization of the exchange relation \eqref{eq:exchange rel}, in which the right-hand side of \eqref{eq:exchange rel} is allowed to be an irreducible Laurent polynomial satisfying some mild conditions. 
We refer to \cite{LP16} for more details.

\medskip

Given an LPA $\mathcal{A}$, we refer to the intersection $\mathcal{U}$ of the Laurent polynomial rings in the cluster variables of any LPA-seed of $\mathcal{A}$ as the \textit{upper} LPA of $\mathcal{A}$. 
By the Laurent phenomenon \cite{LP16}, we have an inclusion $\mathcal{A} \subseteq \mathcal{U}$. 
Although the notion of upper LAPs is a natural generalization of that of upper cluster algebras, it has not yet been extensively studied in the literature.
We refer to \cite{Du22Laurent} for some generalizations of the results of \cite{FZ05} in the context of upper LPAs.

\subsubsection{Mutation semigroup algebras}
In order to state the main definitions of the paper, we fix a finite rank lattice $N$ with dual lattice $M.$

\begin{definition}\label{def:mut-data}
{\em
An {\em (algebraic) mutation datum} is a pair $(u,h=g^k),$ where $u\in N$ is a primitive vector, $g\in \kk[u^\perp \cap M]$ is an irreducible Laurent polynomial, and $k$ is a positive integer.
}
\end{definition}

\begin{definition}\label{def:admiss-mut-data}
{\em
Let $\sigma \subset N_\rr$ be a strongly convex rational polyhedral cone. We say that a mutation datum $(u,h)$ is $\sigma$-{\em admissible} if $u\notin \sigma.$
}
\end{definition}

\begin{definition}\label{def:emb-semigrp}
{\em
Let $\sigma \subset N_\rr$ be a strongly convex rational polyhedral cone. Given a field extension $F$ of $\kk,$ we will refer to a $\kk$-algebra monomorphism $\iota\colon \kk[\sigma^\vee\cap M]\hookrightarrow F$ inducing an isomorphism $\kk(M)\cong F$ as an \textit{embedded semigroup algebra}.
}
\end{definition}

\begin{definition}\label{def:mutation}
{\em
Let $\sigma_1,\sigma_2\subset N_\rr$ be strongly convex rational polyhedral cones, and denote by $U_{\sigma_i}=\Spec \kk[\sigma_i^\vee\cap M]$ the corresponding affine toric varieties. We will call a birational map $\mu\colon U_{\sigma_1}\dashrightarrow U_{\sigma_2}$ a \textit{mutation} if the following conditions are satisfied:
\begin{enumerate}
    \item the induced isomorphism $\mu^*\colon \kk(M)\xrightarrow{\cong}\kk(M)$ of function fields is given on monomials by $\mu^*(x^m)=x^mh^{-\langle u,m\rangle}$ for some $\sigma_2$-admissible mutation datum $(u,h),$ and
    \item strict transform along $\mu$ induces a bijection between the invariant divisors of $U_{\sigma_1}$ and $U_{\sigma_2}$.
\end{enumerate}
If, moreover, we are given a field extension $F$ of $\kk$ and embedded semigroup algebras $\iota_i\colon \kk[\sigma_i^\vee\cap M]\hookrightarrow F,$ we will say that these embedded semigroup algebras \textit{differ by a mutation} if $\iota_1^{-1}\circ \iota_2=\mu^*$ for some mutation $\mu\colon U_{\sigma_1}\dashrightarrow U_{\sigma_2}.$
}
\end{definition}

\begin{definition}\label{def:MSA}
{\em 
A {\em mutation semigroup algebra} is a finitely generated commutative ring $R$ over a field $\kk$ which can be expressed as
\[
R=R_0\cap\hdots\cap R_n,
\]
where the following conditions hold for $i \in \{0,\dots,n\}$:
\begin{enumerate}
\item there exists an embedded semigroup algebra $\iota_i\colon \kk[\sigma_i^\vee\cap M]\hookrightarrow{\rm Frac}(R)$ with $R_i={\rm Im}(\iota_i),$
\item the embedded semigroup algebras $\iota_0$ and $\iota_i$ differ by a mutation, and
\item for each prime ideal $\mathfrak{p}\subset R_i$ of height one, the prime ideal $\mathfrak{p}\cap R$ is of height one in $R$.
\end{enumerate}
In what follows, we abbreviate mutation semigroup algebra as MSA.
We say that a mutation semigroup algebra is simply a {\em mutation algebra} if each $\sigma_i$ is $\{0\}$, i.e., there is no semigroup data on the Laurent polynomial rings.
}
\end{definition} 

\begin{proposition}
\label{prop:cluster are MSA}
Let $t$ be a maximal rank and primitive seed. If the upper cluster algebra $\uclu(t)$ is finitely generated, then it is a mutation algebra. Moreover, if the upper cluster algebra with non-invertible frozen variables $\ouclu(t)$ is finitely generated~\footnote{Finite generation of $\ouclu(t)$ implies finite generation of $\uclu(t)$, as $\uclu(t)$ is the localization of $\ouclu(t)$ at the product of the frozen cluster variables.}, then it is a mutation semigroup algebra.
\end{proposition}

\begin{proof}
Let us first consider the statement for $\uclu(t)$. We set 
\[
R:= \uclu(t), \quad R_0:= \Li(t), \quad R_k:= \Li(\mu_k(t)) \, \quad (k \in I_{uf}).
\]
As the seed $t$ is of maximal rank, Theorem \ref{thm:upper finite} implies that 
\[
R= R_0 \cap \bigcap_{k \in I_{uf}} R_k.
\]
By the Laurent phenomenon, we have that the morphism $\Spec(R_k) \longrightarrow \Spec(R) \, \, (k \in I_{uf} \cup \{ 0 \})$ induced by the natural inclusion $R_k \subseteq R$ is an open embedding. Its image is the locus where the cluster variables of the corresponding seed do not vanish. 
Therefore, condition (3) of Definition \ref{def:MSA} is satisfied.
Let us set $N:= \zz^I$ with canonical basis $(e_i)_{i \in I}$, and denote by $(f_i)_{i \in I}$ its dual basis in the dual lattice $M$ of $N$.
For $k \in I_{uf}$ and $i \in I$, let us set 
\[
\mu_k(f)_i:= \begin{cases} f_i & \mbox{if} \quad i \neq k \\
- f_k + \sum_{j \in I} \max \{ 0, -b_{jk} \} f_j & \mbox{if} \quad i = k.
\end{cases}
\]
One can check that the collection $(\mu_k(f))_{i \in I} \, \, (k \in I_{uf})$  is a basis of the lattice $M.$
Denote by $\iota_k : \kk[M] \longrightarrow {\rm Franc}(R) \, \, (k \in I_{uf} \cup \{0\})$ the homomorphism defined by $x^{\mu_k(f)_i} \longmapsto \mu_k(x)_i \, \, (i \in I)$. 
In the previous definition, we use the convention that $\mu_0(f)_i= f_i$ and $\mu_0(x)_i = x_i \,\,(i \in I).$ 
Notice that, by the Laurent phenomenon, $\iota_k$ is an embedded semigroup algebra. Moreover, we clearly have that ${\rm Im}(\iota_k)= R_k$.
Thus, condition (1) of Definition \ref{def:MSA} is fulfilled.
Finally, we verify that the homomorphisms $\iota_0$ and $\iota_k \, \, (k \in I_{uf})$ differ by a mutation.
This essentially follows from \cite[Section 2]{GHK15birational}, but we include a proof for completeness.

Let us set
\begin{equation}
\label{eq:mut datum cluster}
 u_k:=e_k, \quad v_k:=  \sum_{j \in I}b_{jk}f_j, \quad  g_k=h_k:= 1 + x^{v_k}.
\end{equation}
Then, the mutation formulae for cluster variables \eqref{eq:exchange rel} imply that 
\[
(\iota_0^{-1} \circ \iota_k) (x^m) = x^m h_k ^{-\langle u_k, m \rangle} \qquad (m \in M).
\]
As the seed $t$ is primitive, the proof of \cite[Lemma 4.1]{LP16} implies that $h_k=g_k$ is an irreducible element of $\kk[M]$. 
It moreover belongs to $\kk[u_k^\perp \cap M]$ as $b_{kk}=0.$
Thus, condition (1) of Definition \ref{def:mutation} is satisfied for $\iota_0^{-1} \circ \iota_k.$ 
As there is no torus-invariant divisor in $\Spec(\kk[M])$, condition (2) is also satisfied. Thus, the proof is complete.

We now sketch the proof in the case of $\ouclu(t)$, which is completely analogous to the one of $\uclu(t)$. 
In this case, we set 
\[
\overline{R}:= \ouclu(t), \quad \overline{R}_0:= \overline{\Li}(t), \quad \overline{R}_k:= \overline{\Li}(\mu_k(t)) \, \quad (k \in I_{uf}),
\]
where for a seed $t' \sim t$ with cluster $x'$ we use the notation
\[
\overline{\Li}(t'):= \kk[x'_i \mid  \in I_f][(x'_j)^{\pm 1} \mid j \in I_{uf}].
\]
Moreover, we denote by $\sigma_k \subset N_\rr \, \, (k \in I_{uf} \cup \{0\})$ the cone generated by the elements $e_j$ for $j \in I_f.$
Then, we have that 
\[
\overline{R}= \overline{R}_0 \cap \bigcap_{k \in I_{uf}} \overline{R}_k
\]
and that the restriction $\overline{\iota}_k$ of $\iota_k \, \, (k \in I_{uf} \cup \{0\})$ to $\kk[M \cap \sigma_k^\vee]$ is an isomorphism with image $\overline{R}_k$.
Moreover, the natural morphism $\Spec(\overline{R}_k) \longrightarrow \Spec(\overline{R}) \, \, (k \in I_{uf} \cap \{0 \}$) is an open embedding by the strong version of the Laurent phenomenon \cite[Proposition 11.2]{FZ03}.
Then, notice that the mutation datum of Equation \eqref{eq:mut datum cluster} is clearly $\sigma_k$-admissible, and moreover Condition (2) of Definition \ref{def:mutation} is also obviously satisfied for the rational map between $U_{\sigma_0}$ and $U_{\sigma_k}$ induced by $\overline{\iota}_0^{-1} \circ \overline{\iota}_k$.
Therefore, the discussion of the $\uclu(t)$-case implies that the homomorphisms $\overline{\iota}_0$ and $\overline{\iota}_k \, \, (k \in I_{uf})$ differ by a mutation. 
\end{proof}

\begin{corollary}
Let $\mathcal{A}$ be a locally acyclic cluster algebra admitting a maximal rank and primitive seed~\footnote{This condition is equivalent to requiring that every seed of $\mathcal{A}$ is of maximal rank and primitive.}. Then $\mathcal{A}$ is a mutation algebra.
\end{corollary}

\begin{proof}
As $\mathcal{A}$ is locally acyclic, \cite{Muller13acyclic} implies that $\mathcal{A}$ is equal to its upper cluster algebra $\mathcal{U}$, and that it is finitely generated as a $\kk$-algebra. Then, the statement follows from Proposition \ref{prop:cluster are MSA}.
\end{proof}

\begin{remark}
{\em 
We expect that Proposition \ref{prop:cluster are MSA} can be generalized in various directions for appropriate classes of upper Laurent phenomenon algebras and partial compactifications $\uclu^{\bf v}$ of upper cluster algebras.
}
\end{remark} 

We conclude this subsection by giving an example of a mutation semigroup algebra that is not a 
Laurent phenomenon algebra.

\begin{example}\label{ex:msa-not-lpa}
{\em
Consider the log Calabi--Yau pair $(X,B)$ consisting of the Hirzebruch surface $X=\mathbb{F}_1$ and its toric boundary $B.$ Denote by $N$ (respectively, $M$) the cocharacter (respectively, character) lattice of the acting torus and by $u\in N$ the primitive generator of the ray corresponding to the positive section in $B.$ Consider coordinates $\kk[x^{\pm 1},y^{\pm 1}]$ on $X\setminus B,$ where $x^{\pm 1}\in \kk[u^\perp\cap M]$ are primitive. Let $\beta\colon Y\rightarrow X$ be the blow-up of two distinct non-invariant points on the positive section in $B,$ and denote by $(Y,B_Y)$ the log pullback of $(X,B)$ via $\beta.$ Then $Y$ is a del Pezzo surface and $Y\setminus B_Y\cong\Spec \kk[x^{\pm 1}, y, z]/(yz-(x+\lambda_1)(x+\lambda_2))$ for some $\lambda_1,\lambda_2\in \kk^*.$ In particular, it follows from Theorem ~\ref{introthm:msa-defs-agree} that $R=\kk[x^{\pm 1}, y, z]/(yz-(x+\lambda_1)(x+\lambda_2))$ is an MSA. We claim, however, that algebra is not an LPA.\par
We begin by observing that there are precisely four toric charts $U_1,\hdots, U_4\hookrightarrow Y\setminus B_Y,$ up to isomorphism. Denote by $\pi\colon X\rightarrow \pp^1$ the fibration. Each toric chart $U_i$ can be realized by the blow-downs $\beta_i\colon Y\rightarrow X_i$ of one component of each of the degenerate fibers of $\phi=\pi\circ \beta.$  We may assume that $\beta=\beta_1$ and that $\beta_2\colon Y\rightarrow X_2\cong\mathbb{F}_1$ is the blow-down of the two non-$\beta_1$-exceptional components of the $\phi$-degenerate fibers. For $i=3,4$ we have $X_i\cong \pp^1\times \pp^1.$ The morphism $\phi$ factors through each $\beta_i,$ giving us contractions $\pi_i\colon X_i\rightarrow \pp^1.$ Denoting by $B_i=\beta_{i,*}B_Y$ the toric boundary of $X_i,$ we can choose coordinates $\kk[x^{\pm 1}, y_i^{\pm 1}]$ on $X_i\setminus B_i$ satisfying $y_1=y,$ $y_2=z,$ $y_3=y_1^{-1}(x+\lambda_1)=y_2(x+\lambda_2)^{-1}$ and $y_4=y_1^{-1}(x+\lambda_2)=y_2(x+\lambda_1)^{-1}.$\par
We note that $x^my_i^n\notin R$ for $i=3,4$ and $n\neq 0$. Since $x$ is a unit in $R$ it suffices to show that $y_i^n\notin R$ for $i=3,4$ and $n\neq 0$. We provide an argument in the case $i=3,$ the case $i=4$ being similar. The relation $y_2=y_3(x+\lambda_2)$ gives an injection $R\hookrightarrow S_1=k[x^{\pm 1},y_1,y_3]/(x+\lambda_1-y_1y_3)$, and the relation $y_1=y_3^{-1}(x+\lambda_1)$ gives an injection $R\hookrightarrow S_2=k[x^{\pm 1},y_2,y_3^{-1}]/(x+\lambda_2-y_2y_3^{-1}).$ In the ring $S_1,$ one can check that $y_1$ is prime while $x+\lambda_1$ is not prime. In light of the relation $x+\lambda_1=y_1y_3,$ it follows that $y_3$ cannot be a unit in $S_1$ and hence that $y_3^n\notin R\subset S_1$ for any $n<0.$ Similarly, one sees that $y_3^{-1}$ cannot be a unit in $S_2$ and hence that $y_3^n\notin R\subset S_1$ for any $n>0.$\par
Now, let us assume that $R$ is an LPA. Let $\{z_1,w_1\}$ and $\{z_2,w_2\}$ be the clusters in two seeds differing by a single mutation. Swapping the labeling of these seeds if necessary, we may assume that $\Spec k[z_i^{\pm 1},w_i^{\pm 1}]=U_i$ for $i=1,2.$ Indeed, $\Spec k[z_i^{\pm 1},w_i^{\pm 1}]=U_{j_i}$ for some $1\leq j_i\leq 4,$ and we must have $z_i=\mu_ix^{a_i}y_{j_i}^{b_i}$ and $w_i=\nu_ix^{c_i}y_{j_i}^{d_i}$ for some 
\[
\begin{bmatrix}a_i & b_i\\ c_i & d_i\end{bmatrix}\in {\rm GL}_2(\zz)
\]
and $\mu_i,\nu_i\in \kk^*.$ In particular, at least one of $z_i$ and $w_i$ is of the form $ry_{i_j}^n$ for $r\in R^*$ and $n\neq 0.$ But both $z_i$ and $w_i$ are elements of $R,$ so it follows from the previous paragraph that $1\leq i_j\leq 2.$\par
We may assume $z_1=z_2$ and $w_1=w_2^{-1}z_1^m(z_1+\gamma)$ for some integer $m$ and $\gamma\in \kk^*.$ In the notation of the previous paragraph, we write the equation $z_1=z_2$ as $$\mu_1x^{a_1}y_1^{b_1}=\mu_2x^{a_2}y_1^{-b_2}(x+\lambda_1)^{b_2}(x+\lambda_2)^{b_2}.$$ It follows that we must have $b_1=-b_2=0,$ $a_1=a_2=\pm 1$ and $\mu_1=\mu_2.$ Since 
\[
\begin{bmatrix}a_i & 0\\ c_i & d_i\end{bmatrix}\in {\rm GL}_2(\zz),
\]
it follows that $d_i=a_i$ for $i=1,2.$ Write $a$ for the common value of $a_1=a_2=d_1=d_2$ and $\mu$ for the common value of $\mu_1=\mu_2.$ Replacing $x$ with $\mu x,$ we may assume that $\mu=1.$ We can now write the equation $w_1=w_2^{-1}z_1^m(z_1+\mu)$ as $$\nu_1x^{c_1}y_1^{a}=\nu_2^{-1}x^{am-c_2}y_1^{a}(x+\lambda_1)^{-a}(x+\lambda_2)^{-a}(x^a-\gamma).$$
It follows that $x^a-\gamma$ and $(x+\lambda_1)^a(x+\lambda_2)^a$ differ by units in $k[x^\pm,y_1^\pm],$ which is clearly nonsense.
}
\end{example}

\subsection{Kawamata log terminal singularities}

In this subsection, we recall the definition of some of the singularities of the minimal model program.

\begin{definition}\label{def:klt}
{\em 
A {\em pair} is a couple $(X,B)$ where $X$ is a normal quasi-projective variety and $B$ is an effective divisor on $X$ for which $K_X+B$ is a $\qq$-Cartier divisor. 
Let $\pi\colon Y\rightarrow X$ be a projective birational morphism from a normal variety. 
Let $E$ be a prime divisor on $Y$.
Write 
\[
\pi^*(K_X+B)=K_Y+B_Y,
\]
where $K_Y$ is chosen such that $\pi_*K_Y=K_X$.
The {\em log discrepancy} of $(X,B)$ with respect to $E$ is defined to be 
\[
a_E(X,B):=1-{\rm coeff}_E(B_Y). 
\]
We say that the pair $(X,B)$ is {\em Kawamata log terminal} (resp. {\em log canonical}) 
if $a_E(X,B)>0$ holds for every prime $E$ on every birational moidel $Y$. 
We may write klt to abbreviate Kawamata log terminal
and lc to abbreviate log canonical.
When $B=0$ and the pair $(X,0)$ is klt (resp. lc), 
we may just say that $X$ is klt (resp. lc).
}
\end{definition} 

Roughly speaking, the points in the prime divisor $E$ on the model $Y$ of $X$ correspond to tangent directions through $\pi(E)\subset X$. The log discrepancy $a_E(X,B)$ measures how singular the pair $(X,B)$ is along these tangent directions. The larger the log discrepancies, the milder the singularities of $(X,B)$ are. 
We refer the reader to~\cite{Amb99} for further discussion of log discrepancies.

In dimension two, klt singularities are the same as quotient singularities~\cite[Proposition 9.3]{GKKP11}.
In higher dimensions, every quotient singularity is klt~\cite[Corollary 2.43]{Kol13}.
Elliptic singularities give examples of surface lc singularities which are not klt~\cite[Page 73]{Kol13}.
Two-dimensional pairs $(X,B)$ with log canonical singularities have been classified by Alexeev (see, e.g.,~\cite[Section 3.3]{Kol13}). 
In dimension three, there is no classification of klt singularities; however, when all the log discrepancies of exceptional divisors are larger than one,~\footnote{This case is known as terminal singularities.} these are toric hypersurfaces singularities and are classified by Reid~\cite{Rei87}.

\subsection{Fano varieties} In this subsection, we recall the definition of Fano varieties and some related notions for pairs.

\begin{definition}
{\em 
A {\em Fano variety} is a normal projective variety $X$ for which the anti-canonical divisor $-K_X$ is ample. 
We say that $X$ is a {\em klt Fano variety} 
if it is a Fano variety with klt singularities.
}
\end{definition}

Fano varieties are one of the three building blocks of algebraic varieties. 
Smooth Fano varieties of dimension two are known as {\em del Pezzo surfaces}. 
These were classically classified by the Italian school of algebraic geometers.
On the other hand, klt Gorenstein Fano surfaces were classified by Miyanishi and Zhang~\cite{MZ88}.
Smooth Fano varieties of dimension three were classified by Mori, Mukai, Prokhorov, and Iskovskikh~\cite{MM81,IP99}. In each dimension $n$, there are only finitely many families of smooth Fano varieties. 
This result was proved by Koll\'ar, Miyaoka, and Mori~\cite{KMM92b}.
Furthermore, given $n\in \mathbb{Z}_{>0}$ and $\epsilon>0$, in each dimension $n$ there are only finitely many families of klt Fano varieties with log discrepancies $>\epsilon$. 
This result was conjectured by Borisov, Alexeev, and Borisov, and proved by Birkar in~\cite{Bir21}.
Nowadays, this theorem is known as the boundedness of Fano varieties.

\begin{definition}\label{def:log-Fano}
{\em 
We say that a variety $X$ is {\em log Fano}
if there exists a boundary divisor $B$ on $X$ such that the pair $(X,B)$ is klt 
and the divisor $-(K_X+B)$ is ample.
}
\end{definition} 

In some sense, log Fano varieties can be thought of as Fano varieties up to a perturbation given by the boundary divisor $B$. 
Toric varieties are often not Fano varieties. 
Indeed, a projective toric variety is Fano if and only if its corresponding polytope is reflexive~\cite{CLS11}.
However, every projective toric variety is log Fano. In the case of toric varieties, the divisor $B$ giving the log Fano structure can always be chosen to be torus invariant. 

\subsection{Cluster type varieties} In this subsection, we recall the definition of cluster type pairs and cluster type varieties.
The original definition was given in~\cite{EFM24} using the language of crepant birational maps. 
The following definition is an equivalent one due to~\cite[Theorem 1.3]{EFM24}. 

\begin{definition}\label{def:ct}
{\em 
Let $X$ be a normal quasi-projective variety of dimension $n$. 
We say that $X$ is a {\em cluster type variety} 
if the following conditions are satisfied: 
\begin{enumerate}
\item there is an embedding in codimension one 
$\mathbb{G}_m^n \dashrightarrow X$,
\item the volume form of the algebraic torus
\[
\Omega_{\mathbb{G}_m^n}:=\frac{dx_1\wedge dx_2 \wedge \dots \wedge dx_n}{x_1\cdots x_n} 
\]
has no zeros on $X$. 
\end{enumerate} 
We may write $\Omega_X$ for the induced volume form on $X$. By definition, the volume form $\Omega_X$ only has poles on $X$. 
Given a cluster type variety $X$ and its volume form $\Omega_X$, we denote by $B$ the effective divisor given by the poles of the volume form $\Omega_X$.
}
\end{definition} 

Every toric variety is cluster type.
Indeed, the volume form of the algebraic torus in a toric variety has simple poles at every torus invariant prime divisor.
However, the class of cluster type varieties is much larger. 
For instance, every del Pezzo surface of degree $\geq 2$ is cluster type, and a general del Pezzo surface of degree one is cluster type (see, e.g,~\cite[Theorem 2.1 and Proposition 2.3]{ALP23} and~\cite[Proposition 1.3]{GHK15}).
The projectivization over $\pp^n$ of the tangent bundle $T_{\pp^n}$ is also a cluster type variety~\cite{ELY25}. Many hypersurfaces of low degree in toric varieties are of cluster type as well~\cite{ELY25}. 
By~\cite{MJ24}, the cluster type condition is constructible in families of smooth Fano varieties.

\subsection{Cox rings}\label{def:CR} In this subsection, we recall the concept of the Cox ring of a normal projective variety. 

\begin{definition}
{\em 
Let $X$ be a normal projective variety with finitely generated class group ${\rm Cl}(X)$. 
Fix a subgroup $K\subset {\rm WDiv}(X)$ such that the induced homomorphism $c\colon K\rightarrow {\rm Cl}(X)$ sending $D$ to its class in the Class group $[D]$ is surjective.
Let $K^0\subset K$ be the kernel of $c$.
Fix a group homomorphism $\chi\colon K^0 \rightarrow \kk(X)^*$ 
that satisfies ${\rm div}(\chi(E))=E$ for all $E\in K^0$.
Define the divisorial sheaf $\mathcal{S}$ associated to $K$ to be
\[
\mathcal{S}:=\bigoplus_{D\in K} \mathcal{O}_X(D). 
\]
Consider the ideal sheaf $\mathcal{I}$ defined by 
\[
\langle 1-\chi(E) \mid E\in K^0 \rangle. 
\]
The {\em Cox sheaf} associated to $K$ and $\chi$ is the quotient sheaf $\mathcal{R}:=\mathcal{S}/\mathcal{I}$
together with the induced ${\rm Cl}(X)$-grading.
The {\em Cox ring} of $X$, denoted by ${\rm Cox}(X)$, is the ring of global sections
of its Cox sheaf. 
The {\em Cox space} of $X$, denoted by $\bar{X}$, is the spectrum of the Cox ring of $X$. 
}
\end{definition} 

A priori, the definition of the Cox ring depends on the choice of $K$ and $\chi$. However, for normal projective varieties, all the induced Cox rings are isomorphic graded rings (see~\cite[Proposition 4.2.2]{ADHL15}).
The Cox ring of a toric variety $X$ is a ${\rm Cl}(X)$-graded
polynomial ring known as its homogeneous coordinate ring. 
This ring was studied by Cox (see~\cite{Cox95}). 
Whenever ${\rm Cox}(X)$ is finitely generated, it gives a natural homogeneous coordinates for the variety $X$.
It is often the case that the geometry of $X$ can be understood by studying the Cox ring ${\rm Cox}(X)$.
For instance, a $\qq$-factorial Mori dream space $X$ is log Fano if and only if ${\rm Cox}(X)$ has klt singularities (see~\cite[Theorem 1.1]{GOST15}).
Furthermore, if ${\rm Cox}(X)$ is a polynomial ring, then $X$ is a toric variety~\cite[Corollary 2.10]{HK00}. 

\section{Explicit examples}\label{sec:ex-ex} In this section, we give several examples of the new definitions in this article. We also prove explicit theorems that exemplify the more abstract results of this paper. 

\subsection{Fano compactifications of cluster algebras}
In this subsection, we give explicit examples of Fano and log Fano compactifications of spectra of mutation semigroup algebras.

We start by fixing some notation to give some Lie-theoretic examples.
Let $G$ be a connected reductive algebraic group over $\kk$ and let $X$ be a partial flag variety of $X$. 
In other words, $X$ is a projective variety which is a homogeneous space under a left $G$-action. 
Let $B$ and $B^-$ be opposite Borel subgroups of $G$, and $W$ be the Weyl group of $G$ with respect to the maximal torus $T:=B \cap B^-$. We denote by $\mathcal{S}$ the set of simple reflections of $W$ determined by the Borel group $B$.

\subsubsection{Open Schubert cells and flag varieties}
\label{subsub:Schubert flag}
Let $\Omega \subseteq X$ be the open $B$-orbit. 
The variety $\Omega$ is  an affine space that can be realized as the spectra of a distinguished upper cluster algebra with non-invertible frozen variables \cite{geiss2011kac, goodearl2021integral} closely related to the representation theory of $G$.
The flag variety $X$ is a smooth Fano compactification of $\Omega$.

\begin{example}
\label{ex:schubertSL3}
{
\rm{
Assume, for instance, that $G= \SL_3$ and that $B$ (resp $B^-$) is the Borel subgroup of $G$ consisting of upper (resp. lower) triangular matrices. 
Set $X:=G/B^-$. 
Let 
\[
U:= \left\{ \begin{pmatrix}
1 & a & b \\
0 & 1 & c \\
0 & 0 & 1
\end{pmatrix}  \mid a,b,c \in \kk\right\}
\]
}
be the unipotent radical of $B$.
The morphism $U \longmapsto G/B^-$ defined by $u \longmapsto uB^-/B^-$ is an open embedding whose image is $\Omega$.
Through this isomorphism, the ring $\mathcal{O}(\Omega)$ can be realized as the upper cluster algebra with non-invertible frozen variables $\ouclu(t)$ associated to the seed $t$ whose quiver is 
\[
\begin{tikzcd}
 {\blacksquare 1} & {\bigcirc 2} & {\blacksquare 3}
 \arrow[from=1-1, to=1-2]
 \arrow[from=1-2, to=1-3]
\end{tikzcd}
\] 
and whose cluster variables are
\[
x_1= ac-b, \qquad x_2= a, \qquad x_3= b.
\]

}
\end{example}

\subsubsection{Schubert cells and Schubert varieties}
\label{subsec:schubert}
Assume that $X=G/B^-$ is the \textit{complete} flag variety of $G$. 
Let $X_w:= \overline{BwB^-/B^-} \subseteq X$ be the \textit{(opposite) Schubert variety}\footnote{Throughout this section, we use a notation for Schubert and Richardson varieties that differs slightly from the one of Section~\ref{sec:comb}. The reason is that here we choose the Borel subgroup $B^-$ as the base point of the flag variety $X$, whereas in Section~\ref{sec:comb} the base point is $B$. This choice allows for a simpler comparison with the literature on cluster algebras in the present section, while the choice made in Section~\ref{sec:comb} is better suited for comparing with classical works on the geometry of flag, Schubert, and Richardson varieties.} corresponding to a Weyl group element $w \in W$. 
Moreover, let $\mathring{X}_w:= BwB^-/B^- \subseteq X_w$ be its \textit{(opposite) Schubert cell}. 
As in the previous example, $\mathring{X}_w$ is an affine space that can be realized as the spectra of a distinguished upper cluster algebra with non-invertible frozen variables \cite{geiss2011kac, goodearl2021integral}.
By \cite{AS14}, we have that $X_w$ is a log Fano compactification of $\mathring{X}_w.$

\begin{example}
{
\rm{
Let $G=\SL_4$ and $B$ (resp $B^-$) be the Borel subgroup of upper (resp. lower) triangular matrices. 
In this case, we have that $\mathcal{S}= \{s_1,s_2,s_3\}$ with standard notation \cite{BourbakiGroupes}.
Set $w:=s_2$ and 
\[
U(w_0s_2):= \left\{ \begin{pmatrix}
1 & a & b & c \\
0 & 1 & 0 & d \\
0 & 0 & 1 & e \\
0 & 0 & 0 & 1
\end{pmatrix}  \mid a,b,c,d,e \in \kk\right\}.
\]
The morphism $U(w_0s_2) \longrightarrow \mathring{X}_{s_2}$ defined by $u \longmapsto u s_2 B^-/B^-$ is an isomorphism. 
Through this identification, the ring $\mathcal{O}(\mathring{X}_{s_2})$ can be realized as the upper cluster algebra with non-invertible frozen variables $\ouclu(t)$ associated to the seed $t$ whose quiver is 
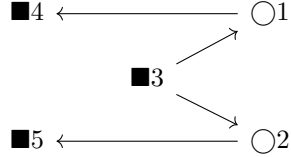
\[
\begin{tikzcd}[row sep=small]
 \blacksquare4 && \bigcirc1 \\
 & \blacksquare3 \\
 \blacksquare5 && \bigcirc2
 \arrow[from=1-3, to=1-1]
 \arrow[from=2-2, to=1-3]
 \arrow[from=2-2, to=3-3]
 \arrow[from=3-3, to=3-1]
\end{tikzcd}
\] 
and whose cluster variables are
\[
x_1= e, \qquad x_2= a, \qquad x_3= ad-c, \qquad x_4=c-ad-be, \qquad x_5=c.
\]
}
}
\end{example}

\subsubsection{Richardson varieties}\label{subsec:richardson}
Assume again that $X=G/B^-$. For a Weyl group element $v \in W$, we denote respectively by $\mathring{X}^v:= B^-vB^-/B^-$ and $X^v:= \overline{B^-vB^-/B^-}$ the \textit{Schubert cell} and the \textit{Schubert variety} associated to $v$.
Moreover, let $\mathring{X}^v_u:= \mathring{X}^v \cap \mathring{X}_u$ be  the \textit{open Richardson variety} associated to two Weyl group elements $v,u\in W$ such that $u \leq v$ in the Bruhat order. 
The last condition on $v$ and $u$ guarantees that  $\mathring{X}^v_u$ is non-empty.
The Richardson varieties are known to be smooth \cite{Richardson92cosets} and to be the spectra of upper cluster algebras by \cite{CGGLSS25, GLSB23, Leclerc16strata}.
We have that  $\mathring{X}^v_u$ is compactified by the \textit{Richardson variety} $X^v_u:=X^v \cap X_u$.
By \cite{KS14}, this yields a log Fano compactification of $\mathring{X}^v_u$.

\begin{example}
{
\rm{
Assume to be in the setting of Example \ref{ex:schubertSL3}, and let $v:=w_0$ and $u:=e.$ 
Then, the open Richardson variety $\mathring{X}^v_u$ is an open subset of  $\Omega:= BB^-/B^-.$
Identifying the latter space with the unipotent group $U$ as in Example \ref{ex:schubertSL3}, we have that
\[
\mathring{X}^v_u = \left\{ \begin{pmatrix}
1 & a & b \\
0 & 1 & c \\
0 & 0 & 1
\end{pmatrix}  \mid b \neq 0 \neq ac-b, \, \, ( a,b,c \in \kk)\right\}.
\]
Thus, the coordinate ring of the variety $\mathring{X}^v_u$ is the upper cluster algebra $\uclu(t)$ associated to  seed $t$ described in Example \ref{ex:schubertSL3}. 
Notice that the frozen cluster variables $x_1$ and $x_3$ are invertible in $\uclu(t)=\oo(\mathring{X}^v_u)$, contrary to the case of $\ouclu(t)=\oo(\Omega)$.
}
}
\end{example}

\subsubsection{Semisimple algebraic groups of type $E_8$, $F_4$, $G_2$ and their wonderful compactification}
Assume that $G$ as above is simple. 
Then, by \cite{FZ05}, the open double Bruhat cell $G^{w_0,w_0}:= Bw_0B \cap B^-w_0B^- \subseteq G$~\footnote{Here $w_0$ denotes the longest element of the Weyl group $W$.} is the spectrum of an upper cluster algebra.
If $G$ is of type $E_8$, $F_4$ or $G_2$, then $G$ is also adjoint \cite{BourbakiGroupes} and it therefore embeds in its \textit{Wonderful compactification} $X$ \cite{DeConcini83symmetric}.
Then, $X$ provides a smooth Fano compactification of $G^{w_0,w_0}.$

\subsubsection{Affinisation of rank 2 cluster varieties}\label{ex:affinisation-rank-2-cluster}

We now turn to the study of compactifications of rank-2 cluster varieties.
By \cite[Corollary 1.21]{FZ05}, spectra of rank 2 upper cluster algebras with no coefficients are either $\mathbb{G}_m^2$ or of the form
\[
\mathcal{A}^{{\rm{af}}}(a,b):= \Spec(\kk[x_1,x_2,x_3,x_4]/ (x_1x_3-1-x_2^a, \, \, x_2x_4-1-x_1^b))
\]
for some strictly positive integers $a,b$. We note that each $\mathcal{A}^{\rm af}(a,b)$ is a smooth affine variety. We provide a klt Fano compactification for each $a,b>0$.
Consider $T:=\pp^1\times \pp^1$.
Let $D_\ell:=\{\ell\}\times \pp^1$
and $B_\ell:=\pp^1\times \{\ell\}$, 
where $\ell\in \pp^1$.
Therefore, the boundary $B_T:=D_0+D_\infty+B_0+B_\infty$ 
is the torus invariant boundary of $T$.
Let $\mathcal{R}_a:=\{1,\mu_a,\mu_a^2,\dots,\mu_a^{a-1}\}$
be the set of $a$-th roots of unity in $D_0$.
Let $\mathcal{R}_b:=\{1,\mu_b,\mu_b^2,\dots,\mu_b^{b-1}\}$
be the set of $b$-th roots of unity in $B_0$.
Let $X_0(a,b)$ be the blow-up of 
$T$ along $\mathcal{R}_a$ and $\mathcal{R}_b$.
Let $B_0(a,b)$ be the strict transform of $B_T$ in $X_0(a,b)$. 
We argue that 
$X_0(a,b)\setminus B_0(a,b)\simeq  \mathcal{A}^{\rm af}(a,b)$.
Indeed, the rational map 
\begin{align} 
\phi\colon  \mathbb{G}_m^2  & \dashrightarrow 
\mathcal{A}^{\rm af}(a,b)\\ 
 (x_1,x_2)  & \mapsto \left(x_1,x_2,\frac{1+x_2^a}{x_1},\frac{1+x_1^b}{x_2}\right) 
\end{align} 
is an isomorphism onto its image and misses precisely the curves defined by 
$\{x_1=x_2^a+1=0\}$
and the curves defined by
$\{x_2=x_1^b+1=0\}$.
Thus, the map becomes an isomorphism after the blow-up, i.e., it induces an isomorphism
$X_0(a,b)\setminus B_0(a,b)\simeq \mathcal{A}^{\rm af}(a,b)$.
Assume that $a+b\geq 3$. 
Write $D'_0$ and $B'_0$ for the strict transforms of $D_0$ and $B_0$, respectively.
Therefore, we have 
\[
(D'_0+B'_0)\cdot D'_0 = 1-a, 
\quad 
(D'_0+B'_0)\cdot B'_0= 1-b, 
\text{ and }
(D'_0+B'_0)^2=2-a-b.
\]
Therefore, by Artin's contractibility criterion, we may contract both curves to a singular point 
$X_0(a,b)\rightarrow X(a,b)$.
The image of $D'_0+B'_0$ is a cyclic quotient singularity $x_0(a,b)$.
Let $B(a,b)$ be the image of $B_0(a,b)$ in $X(a,b)$. 
By construction, we have 
$X(a,b)\setminus B(a,b) \simeq \mathcal{A}^{\rm af}(a,b)$. 
The surface $X(a,b)$ is a projective surface of Picard rank $a+b$.
Furthermore, we have $-K_{X(a,b)}\simeq B(a,b)$.
The pair $(X(a,b),B(a,b))$ has log canonical singularities while the surface $X(a,b)$ has klt singularities. 
The divisor $B(a,b)$ intersects every curve on $X(a,b)$ positively, so $X(a,b)$ is a klt Fano surface. 
We conclude that $(X(a,b),B(a,b))$ is a cluster type pair, $X(a,b)$ is a klt Fano surface, and 
\[
X(a,b)\setminus B(a,b)\simeq \mathcal{A}^{\rm af}(a,b) 
\]
holds. Thus, $X(a,b)$ is a klt Fano compactification of the spectra 
$\mathcal{A}^{\rm af}(a,b)$ of a rank $2$ upper cluster algebra.

We conclude by explaining what happens when $a+b\leq 3$. 
The surface $X(2,1)$ is a smooth del Pezzo surface of degree $6$. The divisor $B(2,1)$ is the strict transform on $X(2,1)$ of the reduced sum of a conic and a line in $\pp^2$ such that the conic contains all the points that are blown-up. 
The surface $X(1,1)$ is a smooth del Pezzo surface of degree $7$. The divisor $B(1,1)$ is the strict transform on $X(1,1)$ of the reduced sum of three lines in $\pp^2$ such that the blown-up points are contained in two of such lines. 

We argue that for every $(a,b)$, the variety $X(a,b)$ is log Fano. Indeed, the divisor 
\[
A_0:= \frac{1}{a}D_0' + \frac{1}{b}B_0'+D_\infty+B_\infty 
\]
is ample, and so $(X_0(a,b),B_0(a,b)-A_0(a,b))$ is a klt Fano pair.
Thus, for each $a,b$, the variety $X(a,b)$ provides a klt Fano compactification of $\mathcal{A}^{\rm af}(a,b)$
while $X_0(a,b)$ provides a canonical log Fano comopactification of $\mathcal{A}^{\rm af}(a,b)$. This provides an explicit example for Corollary~\ref{introcor:comp-locally-acyclic}.

\subsection{Cluster type compactification of mutation semigroup algebras}

In this subsection, we provide explicit examples for Theorem \ref{introthm:msa-defs-agree}.

\subsubsection{From the open Schubert cell to an open Richardson variety}

Assume to be in the setting of Subsection \ref{subsub:Schubert flag} and that $X=G/B^-$ is a complete flag variety of $G$.
 Set 
\[
D^+_s:= \overline{BsB^-/B^-}, \qquad D^-_s:= \overline{B^-sw_0B^-/B^-} \qquad (s \in \mathcal{S}),
\]
and
\[
D^+:= \sum_{s \in \mathcal{S}}D^+_s, \qquad D^-:= \sum_{s \in \mathcal{S}}D^-_s, \qquad D:=D^+ + D^-.
\]
In other words, $D^+$ (resp. $D^-$) is the sum of the $B$-stable (resp. $B^-$-stable) divisors of $X.$
By the Borel-Weil theorem, any divisor $A$ such that $D^+ \leq A \leq D$ is ample.
The fact that the pair $(X,D)$ is cluster type follows from the proof of Corollary~\ref{introcor:var-from-comb-ct}.(3).
Therefore, Theorem \ref{introthm:msa-defs-agree} implies that for any such divisor $A$, the ring $\oo(X \setminus A)$ is a mutation semigroup algebra. 
Let us verify this last assertion by hand.
If $A= D^+$, then $X \setminus D^+= \mathring{X}_e=BB^-/B^-$. Let us identify the ring $\mathcal{O}( X \setminus D^+)$ with the upper cluster algebra with non-invertible frozen variables considered in  \cite{geiss2011kac, goodearl2021integral}, as discussed in Section \ref{subsub:Schubert flag}.
We recall that the operation of tacking the divisor sets a bijection between the set of frozen  variables of this cluster structure and the restriction of the divisors $D^-_s \, \, (s \in \mathcal{S})$ to $X \setminus D^+$.
Therefore, we can denote  by $x_s \, \, (s \in \mathcal{S})$ the frozen cluster variable defined by
$\ddivv(x_s)=(D^-_s)_{| X \setminus D^+}.$
We then have that
\[
\displaystyle \oo\bigl( X \setminus D^+)_{\prod_{s \in \mathcal{S'}} x_s} = \oo\bigl( X \setminus \bigl(D^+ + \sum_{s \in \mathcal{S}'} D^-_s \bigr) \bigr) \qquad (\mathcal{S}' \subseteq \mathcal{S}).
\]
It is well known that localizing an upper cluster algebra with non-invertible frozen variables (such as $\mathcal{O}_X(X \setminus D^+)$) at a subset of its frozen variables yields the same upper cluster algebra, but in which the localized variables are invertible.
Hence, the previous discussion shows that for any $D^+ \leq  A \leq D$, the ring $\oo(X \setminus A, \, \mathcal{O}_X)$ is an upper cluster algebra in which some frozen variables may be non-invertible.
Notice that when $A=D$, we have that $X \setminus D= BB^-/B^- \cap B^-w_0B^-/B^-$ is equal to the open Richardson variety $\mathring{X}^{w_0}_e$, and the previous discussion shows that $\mathcal{O}(X \setminus D)$ is an upper cluster algebra.

\begin{example}
\rm{
Let us make explicit the previous discussion in the setting of Example \ref{ex:schubertSL3}. 
Denote (with standard notation) by $s_1$ and $s_2$ the simple reflections of the Weyl group.
We then have that 
$x_1=x_{s_2}$ and $x_3=x_{s_1}$, where $x_1$ and $x_3$ are the frozen variables of the seed $t$ described in Example \ref{ex:schubertSL3}.
Moreover, for any divisor $D^+ \leq A \leq D$, the coordinate ring of $X \setminus A$ is an upper cluster algebra with non-invertible frozen variables associated to the seed $t$ and in which the frozen cluster variables are invertible or not, depending on $A$. 
In particular, computing that mutating the seed $t$ at the vertex $2$ yields the new cluster variable $x_2'=c$, and using the fact that the seed $t$ is of maximal rank, we get the following list of equalities:
\[
\begin{array}{l}
\oo\bigl( X \setminus D^+ \bigr)= \kk[a^{\pm 1}, b, ac-b ] \cap \kk[c^{\pm 1}, b, ac-b ], \\[0.5em]
\oo\bigl( X \setminus (D^+ + D^-_{s_1}) \bigr)= \kk[a^{\pm 1}, b^{\pm 1}, ac-b ] \cap \kk[c^{\pm 1}, b^{\pm 1}, ac-b ],\\ [0.5em]
\oo\bigl( X \setminus (D^+ + D^-_{s_2}) \bigr)= \kk[a^{\pm 1}, b, (ac-b)^{\pm 1} ] \cap \kk[c^{\pm 1}, b, (ac-b)^{\pm 1}], \\[0.5em]
\oo\bigl( X \setminus (D^+ + D^-_{s_1}+ D^-_{s_2}) \bigr)= \kk[a^{\pm 1}, b^{\pm 1}, (ac-b)^{\pm 1} ] \cap \kk[c^{\pm 1}, b^{\pm 1}, (ac-b)^{\pm 1}].
\end{array}
\]
}
\end{example}

\subsection{Cluster algebra structures on Cox rings}

In this subsection, we study mutation semigroup algebra structures on the Cox rings of projective varieties.

\subsubsection{Minimal desingularisations of cubic surfaces}
\label{subsec:cubic}
In the case of  minimal desingularisations of cubic surfaces in $\pp^3$ with at most rational double points, the Cox ring is known \cite{DHHKL15}. 
Some of these Cox rings are indeed mutation semigroup algebras.
In order to give explicit examples, we fix homogeneous coordinates $Z_0,Z_1,Z_2, Z_3$ on $\pp^3.$

\begin{example}
{
\rm{
Let $Y \subseteq \pp^3$ be the subvariety defined by the homogeneous ideal generated by $Z_0Z_1Z_2-Z_2(Z_1+Z_2)(Z_1+Z_3)$ and $X$ be its minimal desingularisation.
By \cite[Theorem 3.1, (ix)]{DHHKL15}, the Class group of $X$ is isomorphic to $\zz^7$ and the Cox ring $\Cox(X)$ has a presentation given by eleven generators $T_1, \dots T_{11}$ and two relations
\[
T_6T_9-T_1T_2T_7T_8-T_5T_{11}, \qquad T_7T_{10}-T_1T_4T_8^2-T_3T_6T_{11}.
\]
The generators $T_i$ are $\Cl(X)$-homogeneous and their degree is given by the columns of the matrix in \cite[Theorem 3.1, (ix)]{DHHKL15}.
Let $t$ be the seed of the fracrtion field of $\Cox(X)$ defined by the quiver
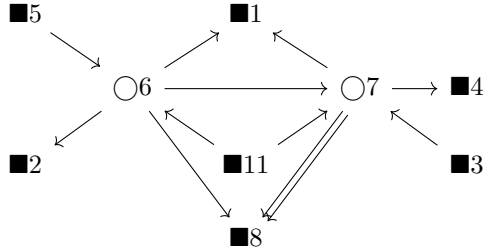
\[
\begin{tikzcd}[sep=scriptsize]
 {{\blacksquare 5}} && {{\blacksquare 1}} \\
 & {{\bigcirc 6}} && {{\bigcirc 7}} & {{\blacksquare 4}} \\
 {{\blacksquare 2}} && {{\blacksquare 11}} && {{\blacksquare 3}} \\
 && {{\blacksquare 8}}
 \arrow[from=1-1, to=2-2]
 \arrow[from=2-2, to=1-3]
 \arrow[from=2-2, to=2-4]
 \arrow[from=2-2, to=3-1]
 \arrow[from=2-2, to=4-3]
 \arrow[from=2-4, to=1-3]
 \arrow[from=2-4, to=2-5]
 \arrow[from=2-4, to=4-3]
 \arrow[shift left, from=2-4, to=4-3]
 \arrow[from=3-3, to=2-2]
 \arrow[from=3-3, to=2-4]
 \arrow[from=3-5, to=2-4]
\end{tikzcd}
\]
and with cluster variables $x_i:=T_i $ for $i \in I:=\{1,2,3,4,5,6,7,8,11\}.$
}
\begin{lemma} 
\label{lem:Cox cubic}
The equalities $\overline{\clu}(t)=\ouclu(t)=\Cox(X)$ hold.
\end{lemma}

\begin{proof}
First, notice that the localisation of $\Cox(X)$ at the product of the variables $T_i \, \, (i \in I)$ is isomorphic to the Laurent polynomial ring in the $T_i \, \, (i \in I)$.
Hence, the cluster variables $x_i \, \, (i \in I)$ are algebraically independent, which implies that the seed $t$ is well defined.
Next, observe that by mutating the seed $t$ at the mutable vertices $6$ and $7$, we get the new cluster variables
\[
\mu_6(x)_6=T_9, \qquad \mu_7(x)_7=T_{10}.
\]
It follows that $\Cox(X)$ is contained in the cluster algebra with non-invertible coefficients $\overline{\mathcal{A}}(t)$.
Moreover, a direct computation shows that the mutable cluster variables of the seed $\mu_7 \circ \mu_6(t)$ are
\[
\mu_7 \circ \mu_6(x)_6=T_9, \qquad \mu_7 \circ \mu_6(x)_7=T_9T_{10}- T_1T_2T_3T_8T_{11}.
\]
Hence, the seeds $t$ and $\mu_7 \circ \mu_6(t)$ are disjoint: in other words, they have no common mutable cluster variable. 
Moreover, the cluster variables of $t$ and $\mu_7 \circ \mu_6(t)$ belong to $\Cox(X)$.
As the  class group $\Cl(X)$ of $X$ is free and finitely generated, the ring $\Cox(X)$ is factorial \cite{Arz09}. 
Then, the statement follows from \cite[Theorem 1.4]{geiss2013factorial}.
\end{proof}
Notice that the cluster algebra $\overline{\clu}(t)=\ouclu(t)$ admits a $\Cl(X)$-grading obtained by declaring that the degree of the cluster variables of $t$ are their degree as homogeneous elements of $\Cox(X)$.
Then, the equality of Lemma \ref{lem:Cox cubic} is an equality of $\Cl(X)$-graded algebras. 
}
\end{example}

\begin{example}
{ 
\rm{
Let $Y \subseteq \pp^3$ be the subvariety defined by the homogeneous ideal generated by $Z_0Z_1Z_2+Z_3^2(Z_1+Z_2+Z_3)$ and $X$ be its minimal desingularisation. 
The $\Cl(X)$-graded ring $\Cox(X)$ is described in \cite[Theorem 3.1, (viii)]{DHHKL15}.
In the notation of \cite{DHHKL15}, let $t$ be the seed of the fraction field of $\Cox(X)$ defined by the quiver 
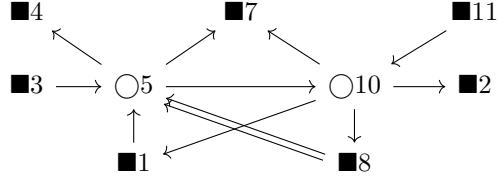
\[\begin{tikzcd}[sep=scriptsize]
 \blacksquare4 && \blacksquare7 && \blacksquare11 \\
 \blacksquare3& \bigcirc5 && \bigcirc10 & \blacksquare2 \\
 & \blacksquare1
    && \blacksquare8
 \arrow[from=1-5, to=2-4]
 \arrow[from=2-1, to=2-2]
 \arrow[from=2-2, to=1-1]
 \arrow[from=2-2, to=1-3]
 \arrow[from=2-2, to=2-4]
 \arrow[from=2-4, to=1-3]
 \arrow[from=2-4, to=2-5]
 \arrow[from=2-4, to=3-2]
 \arrow[from=2-4, to=3-4]
 \arrow[from=3-2, to=2-2]
 \arrow[from=3-4, to=2-2]
 \arrow[shift left, from=3-4, to=2-2]
\end{tikzcd}\]
and with cluster variables $x_i:=T_i$ for $i \in I:= \{1,2,3,4,5,7,8,10,11\}$.
We have an equality of $\Cl(X)$-graded algebras
\[
\overline{\clu}(t)= \ouclu(t)=\Cox(X),
\]
where the cluster and upper cluster algebras of the seed $t$ are graded by assigning to the initial cluster variables the degree given by the corresponding columns of the matrix in \cite[Theorem 3.1, (viii)]{DHHKL15}.
The proof of the previous assertion is completely analogous to that of Lemma \ref{lem:Cox cubic}, and is therefore omitted.
}
}
\end{example}

\subsubsection{Blow-ups of projective spaces}
\label{subsec:blow-up-pn}
We now study the case of blow-ups of the projective space $\pp^n$ at $n+2$ very general points. 
In this case, we discuss two cluster structures.

\medskip

We start by introducing some notation.
Let $n \geq 2$ be a positive integer.
We denote by $\pi: X:=X_n \longmapsto \pp^n$ the blow up of $\pp^n$ at  $n+2$ points $q_i \, \, (1 \leq i \leq n+2)$ in general position. 
Let $Z_0, Z_1, \dots Z_n$ be the homogeneous coordinates of $\pp^n.$
Up to the action of $\PSL_{n+1}$, we can assume that the point $q_i \, \, (1 \leq i \leq n)$ is defined by the equations $Z_0=Z_i$ and $Z_j=0$ for $j \neq i$, and that
\[
 q_{n+1}=[1 : -1 : -1 :  \cdots : -1], \qquad  q_{n+2}=[1: 0 : 0 : \cdots : 0].
\]
Consider the closed irreducible subvariety $E_0:=\pi^{-1}( \{Z_0=0\})$ of $X$ and  denote by $E_i:= \pi^{-1}(q_i) \, \, (1 \leq i \leq n+2)$ the exceptional divisor over $q_i$.
Notice that the divisors $E_i \, \, (0 \leq i \leq n+2)$  freely generate the class group $\Cl(X)$ of $X.$
Let $U \subset X$ be the complement of the support of the $E_i$. The morphism $\pi$ restricts to an isomorphism between $U$ and the complement in $\pp^n$ of the points $q_i$ and of the vanishing locus of $Z_0.$
This allows us to identify $U$ with a big open subset of an affine space of dimension $n$ with coordinates 
\[
z_i:= \frac{Z_i}{Z_0} \qquad (1 \leq i \leq n).
\]
We are now in a position to describe two cluster structures on $\Cox(X)$.

\subsubsection{The Grassmannian cluster structure}
It is well known that the Cox ring of $X$ is isomorphic to the Cox ring of the Grassmannian ${\rm{Gr}}(2,n+3)$ (see, for instance \cite[Remark 3.9]{Castravet2006Hilbert}).
Moreover, it is a foundational result of the theory of cluster algebras that $\Cox({\rm{Gr}}(2,n+3))$ has a distinguished cluster algebra structure.
The seeds of this cluster structure are in bijection with the triangulations of a convex $(n+3)$-gon, and the cluster variables are the Pl\"ucker coordinates. 
We refer the reader to \cite{FZ02} for more details.

\subsubsection{Cluster structure via minimal monomial lifting}

 As the ring $\mathcal{O}(U)$ is a polynomial ring in the variables $z_i$, it admits various structures of cluster algebra. 
 For instance, let us consider the seed $t$ consisting of the quiver:
    \[
    \begin{tikzcd}
 { \bigcirc 1} & { \bigcirc 2} & \cdots & { \bigcirc (n-1)} & { \blacksquare n},
 \arrow[from=1-2, to=1-1]
 \arrow[from=1-3, to=1-2]
 \arrow[from=1-4, to=1-3]
 \arrow[from=1-5, to=1-4]
\end{tikzcd}
\]
and whose cluster $x=(x_1, \dots , x_n)$ is defined as follows. We introduce the auxiliary notation $x_{-1}:= 0$ and $ x_{0}:= 1$, and  we set
\[
x_{k+1}:=z_{k+1}x_k-x_{k-1} \qquad (1 \leq k \leq n).
\]
Then, by \cite[Section 7.1]{geiss2013factorial}, we have that $\overline{\mathcal{U}}(t)= \mathcal{O}(U).$ 
By a general procedure called \textit{Minimal monmomial lifting} \cite{Francone2024Cox}, we can upgrade the cluster structure on $\mathcal{O}(U)$ given by the seed $t$ to a $\Cl(X)$-graded cluster structure in the ring $\Cox(X)$ given by a seed $\lif t$.
For the convenience of the reader, we recall the definition of the seed  $\lif t$ in three steps below.
\begin{enumerate}
\item The set of vertices of $\lif t$ is $\{1,2, \dots, 2n+3\}$.
The vertices of the set $\{1,2, \dots, n-1\}$ are mutable, the remaining ones are frozen.
Heuristically, a vertex $1 \leq i \leq n$ corresponds to the vertex of the seed $t$ with the same label, while a vertex $n+1 \leq j \leq 2n+3$ corresponds to the divisor $E_{j - n+1}$ of $X.$ 
\end{enumerate}
Let $\val_{E_i} : \kk(X)^\times \longmapsto \zz \, \, (0 \leq j \leq n+2)$ be the discrete valuation associating to a rational function on $X$ its order of vanishing along the Cartier divisor $E_j$.
Let $\nu$ be the $(n+3) \times n$ integer matrix given the formula
\begin{equation}
\label{eq:minmatrixBl}
\nu_{ji}:= - \val_{E_{j-1}}(x_i).
\end{equation}
\begin{enumerate}
\item[(2)] For a vertex $i$ of the seed $\lif t$, the corresponding cluster variable $\lif x_i$ is defined as follows:
\[
    \lif x_i := \begin{cases}
        x_i  \in \oo \bigl( Z, \, \,   \mathcal{O} \bigl( \sum_{1 \leq j \leq n+3} \nu_{ji} E_{j-1} \bigr) \bigr) & \text{if} \quad 1 \leq i \leq n, \\[0.5em]
        1 \hspace{0,15cm} \in \oo\bigl( Z, \, \, \mathcal{O}(E_{j-n+1}) \bigr) & \text{if} \quad n+1 \leq i \leq 2n+3.
    \end{cases}
\]
\end{enumerate}
Let $B$ be the generalized exchange matrix of the seed $t$. 
Recall that $B$ is an integer matrix of size $n \times (n-1)$.
\begin{enumerate}
\item[(3)] The exchange matrix $\lif B$ of the seed $\lif t$ is the $(2n+3) \times (n-1)$ integer matrix given by:
\[
\lif B:= \begin{pmatrix}
B\\
- \nu \cdot B
\end{pmatrix}.
\]
\end{enumerate}
Notice that the cluster variables of the seed $\lif t$ are by definition homogeneous elements of the $\Cl(X)$-graded ring $\Cox(X)$. In particular, we have that 
\begin{equation}
\label{eq:degreecoxBL}
{\rm{deg}}(\lif x_i)= \begin{cases}
\sum_{1 \leq j \leq n+3} \nu_{ji}[E_{j-1}] &  \mbox{if} \quad 1 \leq i \leq n,\\[0.5em]
[E_{j-n+1}] & \mbox{if} \quad n+1 \leq i \leq 2n+3,
\end{cases}
\end{equation}
where $[E_j] \, \, (0 \leq j \leq n+2)$ denotes the class of the divisor $E_j$ in $\Cl(X).$
Moreover, the seed $\lif t$ is $\Cl(X)$-graded by the assignment given in Equation  \eqref{eq:degreecoxBL}.
Hence, the upper cluster algebra $\ouclu(\lif t)$ inherits a $\Cl(X)$-grading with respect to which any cluster variable is homogeneous.

\begin{proposition}
\label{prop: cluster cox Bl}
The $\Cl(X)$-graded upper cluster algebra $\ouclu(\lif t)$ is a graded subalgebra of $\Cox(X)$. Moreover, if we denote by $\Sigma$ the product of the frozen cluster variables $\lif x_j \, \, (n+1 \leq j \leq 2n+3)$, then we have that $\ouclu(\lif t)_\Sigma= \Cox(X)_\Sigma$.
\end{proposition}

\begin{proof}
Let us verify that the hypothesis needed to apply \cite[Theorem 3.6]{Francone2024Cox} are fulfilled.
As $U$ is isomorphic to a big open subset of an affine space, we have that $\Pic(U)=\{0\}$ and $\mathcal{O}(U)^\times = K^\times.$
Moreover, $X \setminus U= \bigcup_{0 \leq k \leq n+2} E_k$ is of pure codimension one. 
The  matrix $B$ is clearly of maximal rank.
Hence, the seed $t$ is of maximal rank. 
Then, let us consider two distinct cluster variables $x_i$ and $x_j$ of $t$. 
The functions $x_i$ and $x_j$ are coprime on $\Spec( \mathcal{O}(U))$ by \cite[Lemma 4.0.9]{Francone2023minimal}. 
As the natural morphism $U \longmapsto \Spec( \mathcal{O}(U))= (\pp^n)_{Z_0}$ is an open embedding, we deduce that the functions $x_i$ and $x_j$ are coprime on $U$.
The same argument shows that for any mutable vertex $k$ of $t$, the cluster variables $x_k$ and $\mu_k(x)_k$ are coprime on $U$.
Then, the proposition follows directly from \cite[Theorem 3.6]{Francone2024Cox}.
\end{proof}

Notice that $\ouclu(\lif t)_\Sigma$ is an upper cluster algebra in which the only non-invertible frozen variable is $\lif x_n.$ Hence, Proposition \ref{prop: cluster cox Bl} implies that 
\[
\ouclu(\lif t) \subseteq \Cox(X) \subseteq \uclu(\lif t)_\Sigma.
\]
It is therefore natural to ask if $\Cox(X)$ can be identified with a partial compactification $\uclu(\lif t)^{\bf v}$ of the upper cluster algebra $\uclu(\lif t)$ associated to a collection of tropical points ${\bf v}.$

\section{More preliminaries, propositions, and lemmas}\label{sec:more-prel}

In this section, we collect more technical preliminaries, propositions, and lemmas. For the singularities of the minimal model program, we refer the reader to~\cite{Kol13}. For the classic statements of the minimal model program with scaling, we refer the reader to~\cite{BCHM10}.

\subsection{Birational geometry} In this subsection, we recall some basic definitions from birational models, birational modifications, and Calabi--Yau pairs.

\begin{definition}
{\em 
Let $(X,B)$ and $(Y,B_Y)$ be two pairs. 
We say that $(X,B)$ and $(Y,B_Y)$ are {\em crepant birational equivalent} if the following conditions are satisfied:
\begin{enumerate}
\item there is a birational map $\phi\colon X\dashrightarrow Y$, and 
\item for a resolution of indeterminancy 
$p\colon Z\rightarrow X$ and $q\colon Z\rightarrow Y$ of $\phi$, we have 
\[
p^*(K_X+B)=q^*(K_Y+B_Y).
\]
\end{enumerate} 
}
\end{definition}

\begin{definition}
{\em 
Let $(X,B)$ be a log canonical pair.
We say that $(X,B)$ has {\em divisorially log terminal singularities} if 
there is an open set $U\subset X$ satisfying:
\begin{enumerate}
\item $X$ is smooth on $U$,
\item the divisor $B$ has simple normal crossing singularities on $U$, and 
\item every log canonical center of $(X,B)$ intersects $U$ non-trivially.
\end{enumerate}
We may say that $(X,B)$ has {\em dlt} singularities to abbreviate divisorially log terminal.
}
\end{definition}

\begin{definition}
{\em 
Let $(X,B)$ be a log canonical pair.
A {\em log canonical place} of $(X,B)$ is a prime divisor $E$ on a higher projective birational model 
$\pi\colon Y\rightarrow X$ for which $a_E(X,B)=0$.
A {\em log canonical center} of $(X,B)$ is the image $\pi(E)$ on $X$ of a log canonical place $E$ of $(X,B)$.
}
\end{definition}

\begin{definition}
{\em 
Let $(X,B)$ be a log canonical pair.
We say that $\phi\colon Y \rightarrow X$ is  a {\em divisorially log terminal modification} if the following conditions are satisfied:
\begin{enumerate}
\item every prime divisor $E$ on $Y$ which is exceptional over $X$ is a log canonical place of $(X,B)$; and 
\item if $\phi^*(K_X+B)=K_Y+B_Y$, then $(Y,B_Y)$ is a dlt pair. 
\end{enumerate} 
We may say that $(Y,B_Y)$ is a {\em dlt modification} of $(X,B)$. If furthermore $Y$ is a $\qq$-factorial variety, we may say that $(Y,B_Y)$ is a $\qq$-factorial dlt modification of $(Y,B_Y)$.
}
\end{definition}

The following theorem is well-known to the experts (see, e.g.,~\cite[Theorem 2.9]{FS20}).

\begin{theorem}\label{thm:dlt-mod}
Any log canonical pair $(X,B)$ admits a dlt modification. Furthermore, the dlt modification can be chosen to be $\qq$-factorial.
\end{theorem} 

\begin{definition}
{\em 
A {\em sub-pair} is a couple $(X,B)$ where $X$ is a normal projective variety and $B$ is a divisor on $X$ for which $K_X+B$ is $\qq$-Cartier.
}
\end{definition} 

Definition~\ref{def:klt} of {\em log discrepancies}
for pairs, generalizes verbatim for sub-pairs.
Thus, we may talk about sub-log canonical singularities
and sub-klt singularities for sub-pairs $(X,B)$. 
When $B\geq 0$, this recovers the usual notions explained above.

\begin{definition}
{\em 
A sub-pair $(X,B)$ is said to be {\em sub-log Calabi--Yau} if $K_X+B\sim_\qq 0$ and $(X,B)$ has sub-log canonical singularities. 
If furthermore $B$ is an effective divisor, then we say that $(X,B)$ is a {\em log Calabi--Yau pair}.
}
\end{definition}

\subsection{Fano type varieties} In this subsection, we recall the definition of Fano type varieties that will be used heavily in the proof of the main theorems.
This is a weaker notion than log Fano (see Definition~\ref{def:log-Fano}), which is more suitable for proofs.
We also prove several lemmas about images and birational transformations of Fano type varieties.

\begin{definition}
{\em 
We say that a normal projective variety $X$ is of {\em Fano type} if there exists a boundary $B$ in $X$ such that $(X,B)$ is klt and $-(K_X+B)$ is nef and big.
}
\end{definition}

The definition of Fano type varieties is indeed equivalent to the definition of log Fano varieties. However, in practice, it is easier to show that a given variety is of Fano type.
This equivalence is well-known; we give a short proof. 

\begin{proposition}\label{prop:equiv-FT-logF}
A normal projective variety $X$ is of Fano type if and only if it is log Fano. 
\end{proposition} 

\begin{proof}
It is clear that a log Fano variety is of Fano type. 
Assume that $X$ is of Fano type. 
Then, we can find a boundary $\Delta$ in $X$ such that $(X,\Delta)$ is klt and $-(K_X+\Delta)$ is nef and big. 
By~\cite[Example 2.2.19]{Laz04a}, we can find an effective divisor $E$
such that for every $k\geq 1$, we can write 
\[
-(K_X+\Delta) \sim_\qq A_k + \frac{E}{k},
\]
where $A_k$ is an ample $\qq$-divisor.
By construction, the divisor $E$ is $\qq$-Cartier. 
For $k$ small enough, the pair $(X,\Delta+E/k)$ is klt
and $-(K_X+\Delta+E/k)$ is ample.
Thus, it suffices to take $B:=A_k+\frac{E}{k}$. 
\end{proof} 

The following lemma follows from the definition.
It asserts that Fano type varieties admit some interesting log Calabi--Yau structure.

\begin{lemma}\label{lem:FT-klt-CY}
A normal projective variety $X$ is of {\em Fano type}
if and only if there exists a big boundary divisor $B$ on $X$ such that $(X,B)$ is klt and log Calabi--Yau.
\end{lemma}

The following lemma was proved in~\cite[Corollary 1.3.1]{BCHM10}.

\begin{lemma}\label{lem:FT-are-MDS}
Fano type varieties are Mori dream spaces.
\end{lemma}

The two following lemmas explain how the Fano type condition behaves under birational contractions
and higher birational models (see, e.g.,~\cite[Lemma 2.17]{Mor24a}).

\begin{lemma}\label{lem:FT-bir-cont}
Let $X$ be a Fano type variety. 
Let $X\dashrightarrow Y$ be a birational contraction.
Then $Y$ is a Fano type variety.
\end{lemma}

\begin{lemma}\label{lem:FT-higher-model}
Let $X$ be a Fano type variety.
Let $(X,B)$ be a log Calabi--Yau pair.
Let $\pi\colon Y\rightarrow X$ be a projective birational morphism only extracting log canonical places of $(X,B)$.
Then, the variety $Y$ is of Fano type.
\end{lemma}

\subsection{Cluster type varieties} 

In this subsection, we recall some lemmas regarding cluster type varieties.

\begin{lemma}\label{lem:sing-of-cluster-type}
Let $U$ be a quasi-projective normal variety. 
Assume that there is an embedding in codimension one 
$\mathbb{G}_m^n \dashrightarrow U$.
Let $\Omega$ be the volume form of the algebraic torus $\mathbb{G}_m^n$.
Assume that $\Omega$ has no zeros on $U$
and let $B_U$ be the divisor on $U$ induced
by the poles of $U$.
Then, the pair $(U,B_U)$ has log canonical singularities. 
\end{lemma} 

\begin{proof}
Let $U\hookrightarrow X$ be any embedding of $U$
into a normal projective variety $X$. 
Let $\phi\colon X \dashrightarrow \pp^n$ be a birational map which is an isomorphism over a big open subset of $\mathbb{G}_m^n$. 
Let $p\colon Y\rightarrow X$ and $q\colon Y\rightarrow \pp^n$ be two projective birational morphisms that give a resolution of indeterminacy of $\phi$.
Write $q^*(K_{\pp^n}+\Sigma^n)=K_Y+B_Y$.
Then, by construction, the couple $(Y,B_Y)$ is a sub-log Calabi--Yau pair.
Let $B_X:=p_*B_Y$. Then, by the negativity lemma, the couple $(X,B_X)$ is a sub-log Calabi-Yau pair. In particular, $(X,B_X)$ has sub-log canonical singularities.
Note that all the negative coefficients of $B_X$ happen along prime components of $X\setminus U$. Indeed, the restriction of $B_X$ to $U$ equals $B_U$.
Hence, we conclude that the pair $(U,B_U)$ is log canonical. 
\end{proof} 

\begin{lemma}\label{lem:cluster-type-via-tori}
Let $(X,B)$ be a log Calabi--Yau pair of dimension $n$.
If there exists an embedding in codimension one 
$\mathbb{G}_m^n\dashrightarrow X\setminus B$, then 
the pair $(X,B)$ is of cluster type.
\end{lemma}

\begin{proof}
Let $\mathbb{G}_m^n \subset \pp^n$ be a compactification of the algebraic torus.
Let $\psi\colon \pp^n \dashrightarrow X$ be the induced birational map.
Let $(\pp^n,B_{\pp^n})$ be the sub-log Calabi--Yau obtained by pull-back via $\psi$ to $\pp^n$ of the Calabi--Yau pair $(X,B)$.
As $\mathbb{G}_m^n\dashrightarrow X\setminus B$ is a codimension one embedding, we conclude that $B_{\pp^n}$ is supported on the torus invariant hyperplanes of $\pp^n$. As $K_{\pp^n}+B_{\pp^n}\sim 0$ and the coefficients of $B_{\pp^n}$ are at most one, we conclude that $B_{\pp^n}=H_0+\dots+H_n$, where the $H_i$'s are the hyperplane coordinates. 
We conclude that $(\pp^n,H_0+\dots+H_n)$ is crepant birational equivalent to $(X,B)$.
In particular, the divisor $B$ gives the zeros and poles of the volume form $\Omega_{\mathbb{G}_m^n}$ of the algebraic torus $\mathbb{G}_m^n$. 
Thus, the volume form $\Omega_{\mathbb{G}_m^n}$ has no zeros on $X$, and so $X$ is a cluster type variety and $(X,B)$ is a cluster type pair.
\end{proof} 

\begin{lemma}\label{lem:torus-in-toric}
Let $T$ be a toric variety of dimension $n$.
Let $P\subset T$ be a torus invariant prime divisor.
Then, there exists an embedding $\mathbb{G}_m^n \hookrightarrow T$ that contains the generic point of $P$ and such that 
$T\setminus j(\mathbb{G}_m^n)$ contains all prime torus invariant divisors of $T$ except for $P$.  
\end{lemma} 

\begin{proof}
Let $Q_T$ be the sum of all the prime torus invariant divisors of $T$ distinct from $P$. 
Then, we have that $T\setminus Q_T \simeq \mathbb{A}^1 \times \mathbb{G}_m^{n-1}$ (see, e.g.,~\cite[Theorem 3.2.6]{CLS11}). 
Further, the divisor $P$ is identified with 
$\{0\} \times \mathbb{G}_m^{n-1}$ under the previous isomorphism. 
Consider the open set $U:=(\mathbb{A}^1-\{1\})\times \mathbb{G}_m^{n-1} \simeq \mathbb{G}_m^n$ and $j\colon U\hookrightarrow T$ its embedding. Then, $j(U)$ contains the generic point of $P$ and $T\setminus j(U)$ contains all prime torus invariant divisors of $T$ except for $P$.
\end{proof}

\begin{lemma}\label{lem:toric-model-cluster-type}
Let $(X,B)$ be a cluster type pair.
Let $P$ be a prime component of $B$.
Then, there exists a crepant birational map 
$(T,B_T)\dashrightarrow (X,B)$ satisfying the following conditions:
\begin{enumerate}
\item the pair $(T,B_T)$ is toric, 
\item the center $P_T$ of $P$ in $T$ is a prime divisor,
\item for every prime component $Q\subseteq \supp(B)$ the center
of $Q$ on $T$ is a prime divisor, and 
\item for every prime divisor $Q\subset X$ with $Q\not\subset \supp(B)$ either 
\begin{enumerate} 
\item[(a)] the center of $Q$ in $T$ is a prime divisor, or 
\item[(b)] the center of $Q$ in $T$ is a codimension two subset
which is not contained in $P_T$. 
\end{enumerate} 
\end{enumerate} 
\end{lemma} 

\begin{proof}
Let $(T_0,B_{T_0})\dashrightarrow (X,B)$ be any crepant birational map from a toric log Calabi--Yau pair. 
Possibly replacing $(T_0,B_{T_0})$ with a dlt modification, 
we may assume that for every prime divisor $Q$ in $X$, the center of $Q$ in $T$ is either a prime divisor or a codimension two subvariety contained in a unique torus invariant divisor.
We may further assume that for every prime component $Q\subset \supp(B)$ the center of $Q$ on $T$ is a prime divisor. 
We let $P_{T_0}$ be the center of $P$ on $T_0$ which is a prime torus invariant divisor.
Let $Q_1,\dots,Q_k \subset X$ with $Q_i\not \subset \supp(B)$ be the prime divisors for which 
$c_T(Q_i)$ is contained in $P_{T_0}$.
By possibly replacing $(T_0,B_{T_0})$ with a higher dlt modification, we may assume that there exists a $\pp^1$-fibration $\pi\colon T_0\rightarrow Z$ for which $P_{T_0}$ is a torus invariant section. 
Let $Q_{Z,1},\dots,Q_{Z,k}$ be the images of $c_{T}(Q_1),\dots,c_T(Q_k)$ in $Z$, respectively.
Let $F_1,\dots,F_k$ be the fibers of $Q_{Z_1},\dots,Q_{Z,k}$ via $\pi$, respectively.
Let $Y\rightarrow T_0$ be a blow-up of $T$ extracting the divisors $Q_1,\dots,Q_k$. 
Let $B_Y$ be the strict transform of $B_{T_0}$ in $Y$. 
Let $P_Y$ be the center of $P$ in $Y$.
Then, the pair $(Y,B_Y)$ is a log Calabi--Yau pair.
Let $F_{Y,1},\dots,F_{Y,k}$ be the strict transforms of $F_1,\dots,F_k$ in $Y$. 
Then, for $\epsilon>0$ small enough, the pair
$(Y,B_Y+\epsilon(F_{Y,1}+\dots+F_{Y,k}))$ is dlt. 
We run a $(K_Y+B_Y+\epsilon(F_{Y,1}+\dots+F_{Y,k}))$-MMP 
over $Z$ that terminates after contracting all the divisors $F_{Y,1},\dots,F_{Y,k}$ (see, e.g.,~\cite[Lemma 2.9]{Lai11}).
A complexity computation shows that this MMP must terminate with a toric log Calabi--Yau pair $(T,B_T)$. By construction, conditions $(1)$-$(3)$
are still satisfied by the pair $(T,B_T)$.
Let $P_T$ be the center of $P$ in $T$.
Note that the divisors $F_{Y,1},\dots,F_{Y,k}$ are disjoint from $P_Y$, therefore, after they are contracted, their centers are not contained in the strict transform of $P_Y$.

Let $Q$ be a prime divisor on $X$. 
If $Q$ is a prime component of $B$, then the center of $Q$ in $T_0$ is a prime divisor, and so it is a prime divisor in $T$. 
If $Q$ is a prime divisor in $X$ with $Q\not\subset \supp(B)$, then either its center in $T_0$ is a prime divisor or a codimension two subset which is contained in  a unique prime torus invariant divisor.
Let $Q$ be a prime divisor on $X$ whose center on $T$ has codimension two. 
There are two options: either $c_Q(T_0)$ is a prime divisor, or it has codimension two. 
If $c_Q(T_0)$ has codimension two and $Q=Q_i$ then $c_{Q_i}(T)$ is a prime divisor.
If $c_Q(T_0)=F_i$ for some $i$, then the center of $Q$ in $T$ is contained in a unique prime component of $B_T$, different from $P_T$. 
Finally, if $Q\neq Q_i$ and $c_Q(T_0)\neq F_j$, then 
$c_Q(T)$ is a subvariety of codimension two which is not contained in $P_T$. This finishes the argument of (4).
\end{proof} 

\begin{lemma}\label{lem:affine-partial-compactifications}
Let $(X,B)$ be a $n$-dimensional log Calabi--Yau pair, and let $j\colon \mathbb{G}_m^n\dashrightarrow X\setminus B$ be an open embedding in codimension one. Assume that there exists an ample divisor $0\leq A \leq B.$. Then there exists an affine toric variety $V$ and a morphism $\widetilde{j}\colon V \rightarrow X\setminus A$, unique up to unique isomorphism, satisfying the following conditions:
\begin{enumerate}
    \item $\widetilde{j}|_{\mathbb{G}_m^n}=j,$
    \item strict transform along $\widetilde{j}$ induces a bijection between the invariant divisors on $V$ and the components of $B\setminus A$,
    \item no divisor with generic point in $X\setminus \widetilde{j}(V)$ has center on $V.$
\end{enumerate}
In particular, $\widetilde{j}$ is an open embedding in codimension one.
\end{lemma}
\begin{proof}
There exists a commutative diagram
\[
\xymatrix{ 
(Y,B_Y) \ar[d]_-{f}\ar@{-->}[rd]^-{\pi} & \\
(X,B) & (T,B_T) \ar@{-->}[l]^-{\phi} 
}
\]
satisfying the following conditions:
\begin{enumerate}
    \item[(i)] $(T,B_T)$ is a simplicial projective toric log Calabi--Yau pair on which every component of $B$ has divisorial center,
    \item[(ii)] $\phi$ is the crepant birational map satisfying $\phi|_{\mathbb{G}_m^n}=j,$
    \item[(iii)] $f$ is a $\qq$-factorial dlt modification, and
    \item[(iv)] $\pi$ is a contracting birational map.
\end{enumerate}
Denote by $A_Y=f^*A,$ and let $A_T$ be the strict transform on $T$ of $A_Y.$ Since the divisor $A_Y$ is big and semiample, it follows from ~\cite[Lemma 3.39]{KM98} and condition (iv) above that $A_T$ is big and movable. Replacing $T$ with a small $\qq$-factorial modification, we may assume that $A_T$ is big and semiample. Let $h\colon (T,B_T)\rightarrow (S,B_S)$ be the ample model of $A_T$ and write by $A_S=h_*A_T.$ Thus, $h$ is a toric birational morphism and $0\leq A_S\leq B_S$ is an ample divisor satisfying $A_T=h^*A_S.$\par
We claim that the induced birational map $\psi=\phi\circ h^{-1}\colon S\setminus A_S\dashrightarrow X\setminus A$ is a morphism. Let 
\[
\xymatrix{ 
 & Z \ar[dl]_{p} \ar[dr]^q&\\
Y & & S 
}
\]
be a resolution of the indeterminacy of $h\circ \pi.$ Denote by $E\subset Z$ the $p$-exceptional locus and by $F\subset Z$ the $q$-exceptional locus. The subset $E$ is purely divisorial, as $Y$ is $\qq$-factorial by (iii). It follows from this and (iv) that $E\subset F$. Thus, the big open $S\setminus q(F)$ in $S$ is isomorphic to the open subset $Y\setminus p(F)$ of $Y.$ It follows that $S\setminus (A_S\cup q(F))$ is isomorphisc to $Y\setminus(A_Y\cup p(F)).$ Since $A_Y=f^*A$ and $Y\setminus(A_Y\cup p(F))\subset Y\setminus A_Y,$ it follows that we have a morphism $S\setminus (A_S\cup q(F))\rightarrow X\setminus A.$ But $S\setminus(A_S\cup q(F))$ is a big open in the normal affine variety $S\setminus A_S$ and $X\setminus A$ is affine, so this morphism factors through a morphism $S\setminus A_S\rightarrow X\setminus A.$ \par
It follows from the fact that $S\setminus A_S\rightarrow X\setminus A$ is a morphism that every component of $B\setminus A$ has divisorial center on $S\setminus A.$ Indeed, $h$ restricts to a morphism $T\setminus A_T\rightarrow S\setminus A_S$ since $A_T=h^*A_S$. Thus, the restriction to $T\setminus A_T$ of $\phi$ is a morphism that factors as a composite of two morphisms $$T\setminus A_T\xrightarrow{h} S\setminus A_S \xrightarrow{\psi}X\setminus A.$$
Let $D$ be a component of $B\setminus A$. It follows from (i) that $D$ has divisorial center on $T\setminus A_T,$ say $D_T\subset T\setminus A_T,$ and $\phi(D_T)$ is dense in $D.$ Denoting by $D_S$ the closure in $S\setminus A_S$ of $h(D_T),$ it follows from the factorization $\phi=\psi\circ h$ above that $\psi(D_S)$ is dense in $D.$ Thus, $D_S$ must be a divisor on $S\setminus A_S.$ \par
Now let $V$ be the affine toric variety corresponding to the subcone of that of $S\setminus A_S$ spanned by the rays corresponding to the components of $B\setminus A_S.$ We have a toric morphism $V\rightarrow S\setminus A_S,$ hence a morphism $V\rightarrow X\setminus A.$ This is the desired morphism $\widetilde{j}.$ It is clear that (1) and (2) together imply that $\widetilde{j}$ is an open embedding in codimension one. To show the asserted uniqueness, let us suppose that we are given another such $\widetilde{j}'\colon V'\rightarrow X\setminus A.$ We obtain a birational map $\widetilde{j}^{-1}\circ \widetilde{j}'\colon V'\dashrightarrow V$ between normal affine varieties which is an isomorphism in codimension one. It follows that $\widetilde{j}^{-1}\circ \widetilde{j}'$ is a biregular isomorphism.\par
Finally, we show that condition (3) holds for the morphism $\widetilde{j}$. Choose a toric compactification $W$ of $V,$ and denote by $B_W$ the toric boundary of $W$ and by $A_W$ the reduced sum of those components of $B_W$ lying in $W\setminus V.$ Denote by $\theta\colon (W,B_W) \dashrightarrow (X,B)$ the crepant birational map with $\theta|
_V=\widetilde{j}.$ Let $D\subset X$ be a prime divisor with generic point in $X\setminus \widetilde{j}(V).$ If $D$ is not $\theta^{-1}$-exceptional, then it follows from condition (2) that the center on $W$ of $D$ is a component of $A_W$. Thus, we may assume that $D$ is $\theta^{-1}$-exceptional. In this case, the center on $W$ of $D$ cannot meet the domain of $\theta.$ But $V$ is contained in the domain of $\theta.$ 
\end{proof}

\section{Varieties from Lie theory are cluster type}\label{sec:comb}
In this section, we prove that many varieties coming from Lie theory are indeed of cluster type. We refer the reader to~\cite{Brio05} and~\cite{BK05} for background on the generalized flag variety.

Let $G$ be a semisimple algebraic group, with Borel and opposite Borel subgroup $B,B^-\subset G$. We use the notation $\Lambda \supset \Lambda^+ \supset \Lambda^{++}$ for the weight lattice, the pointed cone of dominant weights, and its interior consisting of the regular dominant weights. We denote the roots $\Phi\subset \Lambda$, and the positive roots $\Phi^+\subset \Phi$, and write $\omega_1,\ldots,\omega_n$ for the fundamental dominant weights that generate $\Lambda^+$.

The \emph{generalized flag variety} $G/B$ has the structure of an algebraic variety as follows: given a dominant weight $\lambda\in \Lambda^{+}$, there is a highest weight representation $V_\lambda$, with highest weight vector $v_\lambda$, and $B$ is contained in the stabilizer of $\langle v_\lambda \rangle\subset \mathbb{P}(V_\lambda)$. This induces a map $G/B\to \mathbb{P}(V_\lambda)$, and when $\lambda\in \Lambda^{++}$ this is an injection which induces an algebraic structure on $G/B$.

The Picard group of $G/B$ is identified with $\Lambda$, and we write $\mathcal{L}_\alpha$ for the line bundle associated to $\alpha\in \Lambda$. 
A line bundle $\mathcal{L}_\lambda$ on $G/B$ is (very) ample if and only if $\lambda\in \Lambda^{+}$ (resp. $\Lambda^{++}$), and in this case it corresponds to the morphism $G/B\to \mathbb{P}(V_\lambda)$ described above.


There is a distinguished regular dominant integral weight $\rho=\sum_{i=1}^n \omega_i=\frac{1}{2}\sum_{\alpha\in \Phi^+}\alpha\in \Lambda^{++}$ defined as either the sum of the dominant weights, or the half-sum of the positive roots. It is well known that $G/B$ is a Fano variety with anticanonical divisor $-K_{G/B}=\mathcal{L}_{2\rho}$.

Let $W$ be the Weyl group of $G$ with respect to the maximal torus $T:=B \cap B^-$, and $\ell : W \longmapsto \nn$ be the length function determined by the inclusion $T \subset B$.
For $v,u \in W$, the corresponding Schubert and opposite Schubert cells are $\mathring{X}^v=BvB/B$ and $\mathring{X}_u=B^-uB/B$. If $u\le v$ in the Bruhat order, the open Richardson variety is defined to be $\mathring{X}^v_u=\mathring{X}^v\cap \mathring{X}_u$. In fact, this variety is empty if $u\not\le v$, so this may be used as the definition of the Bruhat order.

The Schubert variety and opposite Schubert varieties are defined to be $X^v=\overline{\mathring{X}^v}=\overline{BvB}/B$ and $X_u=\overline{\mathring{X}_u}=\overline{BuB}/B$. For $u\le v$ the Richardson variety is defined to be $X^v_u=\overline{X^v_u}=X^v\cap X_u$. Clearly the Richardson variety is a compactification of the open Richardson variety. Let $\partial X^v_u=X^v_u\setminus \mathring{X}^v_u$. Denoting $x\prec y$ to mean $x\le y$ in the Bruhat order and $\ell(y)=\ell(x)+1$, we have
\begin{equation}\label{eq:bruhat-order} 
\partial X^v_u=X^v_u\setminus \mathring{X}^v_u=\left(\bigcup_{u\prec u'\le v}X^v_{u'}\right)\cup \left(\bigcup_{u\le v'\prec v}X^{v'}_u\right).
\end{equation} 

\begin{theorem}\label{thm:cluster-type-x^v_u}
The pair $(X^v_u,\partial X^v_u)$ is a cluster type pair.
\end{theorem} 

\begin{proof}  
We have $c_1(-K_{G/B})=[\partial X^v_u]=\left(\bigcup_{u\prec u'\le v}X^v_{u'}\right)\cup \left(\bigcup_{u\le v'\prec v}X^{v'}_u\right)$ by~\cite[Theorem 4.2.1]{Brio05}. Thus, it suffices to show that 
$(X_u^v,\partial X_u^v)$ has log canonical singularities.
Each Richardson variety is normal and irreducible by \cite[Corollary 4.7]{KLS14}. Then, the statement follows from equality~\eqref{eq:bruhat-order}, induction on the dimension, and inversion of adjunction~\cite{Hac14}.
The fact that $X_u^v\setminus \partial X_u^v$ contains an algebraic torus is standard, see, e.g.~\cite{BFZ05}.
\end{proof} 

We can in fact identify a precise crepant resolution of the pair $(X^v_u,\partial X^v_u)$ that has appeared in the connections between open Richardson varieties and cluster algebras. Let $w_o$ be the longest element of the Weyl group. We follow~\cite[Appendix A]{KLS14}. Fix reduced words $v=s_{i_1}\cdots s_{i_{r}}$ and $w_ou=s_{i_{r+s}}\cdots s_{i_{r+1}}$. Then we define 
$$Z^{v}_{u}=\{(B=g_1B,\ldots,g_{r+s+1}B=w_oB) \mid g_{j}^{-1}g_{j+1}\in  X^{s_{i_j}}\}\subset (G/B)^{r+s+1}.$$
There is a map $Z^v_u\to X^v_u$ given by  projection onto $g_{i_r}B$, and defining $\partial Z^v_u$ to be the locus where $g_iB=g_{i+1}B$ for some $i$, this restricts to an isomorphism
$$Z^v_u\setminus \partial Z^v_u\to X^v_u\setminus \partial X^v_u.$$

The compactification $Z^v_u$ is known as the \emph{brick compactification} of $\mathring{X}^v_u$.

\begin{theorem}\label{thm:crepant-resol}
The morphism $\pi\colon (Z_u^v,\partial Z_u^v)\rightarrow 
(X_u^v,\partial X_u^v)$ is a crepant log resolution of singularities. 
Furthermore, the divisor $\partial Z_u^v$ supports an effective ample divisor.
\end{theorem} 

\begin{proof} 
It is known that $Z^v_u$ is smooth with $\partial Z^v_u$ a simple normal crossings divisor (see~\cite[Subsection 4.2]{Brio05}), and  $[\partial Z^v_u]=-K_{Z^v_u}$ (see, e.g.,~\cite[Lemma A.2]{KLS14}).
Thus, we conclude that $(Z_u^v,\partial Z_u^v)$ is a log smooth log Calabi--Yau pair.
The fact that $\partial Z_u^v$ supports an ample effective divisor is proved in~\cite[Lemma A.6]{KLS14}.
By Theorem~\ref{thm:cluster-type-x^v_u}, we know that $(X_u,\partial X_u^v)$ is log Calabi--Yau.
By construction, we know that $\pi_*\partial Z_u^v= \partial X_u^v$. Therefore, the projective birational morphism $\pi$ is a crepant birational map between the log Calabi--Yau pairs.
\end{proof} 

\begin{proof}[Proof of Theorem~\ref{introthm:brick-richardson}]
The statements (1)-(4) are contained in Theorem~\ref{thm:crepant-resol}.
\end{proof} 

\begin{proof}[Proof of Corollary~\ref{introcor:var-from-comb-ct}]
All the varieties in (1)-(6) are isomorphic to either $X_u^v$ or $Z_u^v$ above.
As $X_u^v\setminus \partial X_u^v \simeq Z_u^v\setminus \partial Z_u^v$, we conclude that 
$Z_u^v$ is of cluster type. 
As $\partial Z_u^v$ supports an effective ample divisor, we conclude that $Z_u^v$ and so $X_u^v$ are Fano type.
\end{proof}

\begin{remark} 
{\em 
     Given any word $\underline{w}=(s_{i_1},\ldots,s_{i_{p}})$  we may similarly define
     $$Z(\underline{w})=\{(B=g_1B,\ldots,g_{p}B=w_oB) : g_{j}^{-1}g_{j+1}\in  X^{s_{i_j}}\}\subset (G/B)^{p},$$
     and the boundary $\partial Z(\underline{w})$ is the subset where $g_iB=g_{i+1}B$ for some $i$. Then $\mathring{X}(\underline{w})=Z(\underline{w})\setminus \partial Z(\underline{w})$ is a smooth variety known as a \emph{braid variety} (see for example~\cite{CGGS24,GLSB23,CGGLSS25}) and $Z(\underline{w})$ is a simple normal crossings compactification of $\mathring{X}(\underline{w})$ known as the brick compactification. Our results would extend to arbitrary brick compactifications and braid varieties if we knew that $\partial Z(\underline{w})$ supported an effective ample and $[\partial Z(\underline{w})]=-K_{Z(\underline{w})}$. We anticipate this result would follow by carrying out the analogous Frobenius splitting arguments from~\cite[Appendix A]{KLS14}, but we have not verified this.}
\end{remark} 

\section{Geometry of mutation semigroup algebras}
\label{sec:geom}
In this section, we prove a geometric characterization of 
mutation semigroup algebras 
with mild singularities. 
More precisely, we prove the following theorem.

\begin{theorem}\label{thm:geom-msa}
Let $\kk$ be an algebraically closed field of characteristic zero.
Let $R$ be a finitely generated commutative ring over $\kk$.
Assume that $U:={\rm Spec}(R)$ has klt singularities. 
Then, the following statements are equivalent:
\begin{enumerate}
\item The ring $R$ is a mutation semigroup algebra over $\kk$. 
\item There is an open embedding in codimension one 
$\mathbb{G}_m^n \dashrightarrow U$ such that the volume form $\Omega$ of $\mathbb{G}_m^n$ vanishes nowhere on $U$.  
\item There exists a cluster type pair $(X,B)$
and an ample divisor $0\leq A \leq B$ such that
$R\simeq \Gamma(X\setminus A, \mathcal{O}_X)$.
\end{enumerate}
\end{theorem}

\begin{lemma}\label{lem:principalization-on-sing-var}
Let $X$ be a normal projective variety.
Let $\mathcal{I}$ be an ideal sheaf on $X$ which is locally principal on a big open subset $U$. 
Then, there exists a projective birational morphism
$\phi\colon Y\rightarrow X$, which is a sequence of blow-ups of $X$ along subvarieties of codimension at least two, such that $\phi^{-1}\mathcal{I}$ is locally principal in $Y$. Furthermore, 
we may choose all the centers to be disjoint from $U$.
\end{lemma}

\begin{proof}
First, let $\psi\colon Z\rightarrow X$ be the blow-up of $\mathcal{I}$. 
Then, by definition $\psi^{-1}\mathcal{I}$, is a locally principal ideal.
Furthermore, the projective morphism $\psi$
is an isomorphism, and hence flat, over $U$.
Let $Z_n \rightarrow X_n$ be the Gruson-Raynaud flattening on $Z\rightarrow X$. 
By~\cite[Theorem A]{Ryd25}, we know that $X_n\rightarrow X$ can be chosen to be a sequence of blow-ups along codimension two subvarieties.
Furthermore, by~\cite[Theorem A]{Ryd25}, we know that
the centers of $X_n\rightarrow X$ are contained in $X\setminus U$.
As $Z_n\rightarrow X_n$ is a flat projective birational morphism, we conclude that it is an isomorphism.
Hence, it suffices to take $Y:=Z_n$ in the statement. 
\end{proof} 

\begin{proof}[Proof of Theorem~\ref{thm:geom-msa}]
First, we argue that (1) implies (2). We claim that each morphism $\phi_i\colon U_i:=\Spec R_i\rightarrow U$ is an open embedding in codimension one and that the union of their images meet every prime divisor in $U.$ Granting this claim, we explain how (2) follows. It is immediate that the restriction of each $\phi_i$ to the open torus $\mathbb{G}^n_m\subset U_i$ is an embedding in codimension one. For each $0\leq i \leq n,$ let $W_i\subset U_i$ be a big open subset on which $\phi_i$ restricts to be an open embedding. It follows from Definition~\ref{def:MSA}.(2) that we may glue together volume forms on the $W_i\cap \mathbb{G}_m^n$ to obtain a rational volume form that vanishes nowhere on $W_0\cup\hdots\cup W_n$. But it follows from the claim that $W_0\cup\hdots\cup W_n$ meets every prime divisor on $U,$ hence this rational volume form vanishes nowhere on $U.$\par
To complete the proof that (1) implies (2), we prove our claim. It follows from Definition ~\ref{def:MSA}.(3) that for each $0\leq i \leq n$ and each prime divisor $D\subset U_i$, the closure in $U$ of $\phi_i(D)$ is a divisor. It follows, for each $0\leq i \leq n$, that there exists a big open $W_i\subset U_i$ on which $\phi_i$ restricts to be quasi-finite. Since $U$ is normal and each $\phi_i$ is birational, it follows from Zariski's Main Theorem that each $\phi_i|_{W_i}$ is actually an open embedding. To show that the open subset $W_0\cup\hdots\cup W_n\subset U$ meets every prime divisor in $U,$ we will assume that there is a prime divisor $D\subset U$ lying outside this union and will obtain a contradiction. Indeed, it follows from this assumption that $\Gamma(U\setminus D,\mathcal{O}_U)\subset \Gamma(W_i,\mathcal{O}_U)=R_i$ for each $0\leq i \leq n.$ Thus, $$R\subsetneq\Gamma(U\setminus D,\mathcal{O}_U)\subset R_1\cap\hdots\cap R_n.$$
But $R=R_1\cap\hdots\cap R_n$ by Definition ~\ref{def:MSA}, so this is nonsense.\par
Now, we turn to prove that (2) implies (3).
Let $U$ be a klt affine variety 
for which there exists an embedding in codimension one
$\mathbb{G}_m^n\dashrightarrow U$ and the volume form
$\Omega$ of the algebraic torus $\mathbb{G}_m^n$
is nowhere vanishing on $U$.
Let $B_U$ be the divisor induced by the poles of $\Omega$ on $U$.
We know that the pair $(U,B_U)$ is log canonical (see Lemma~\ref{lem:sing-of-cluster-type}).
As $U$ is klt, we know that $B_U$ is a $\qq$-Cartier divisor. 
Furthermore, the divisor $B_U$ contains on its support all the log canonical centers of the pair $(U,B_U)$.
Choose a projective closure $U\hookrightarrow X_0$
such that $X_0\setminus U$ supports a very ample divisor $A_{X_0}$.
By normalizing $X_0$, we may assume it is normal.
Let $B_{X_0}$ be the closure of $B_U$ on $X_0$. 
We choose $m\in \mathbb{Z}_{\geq 0}$ such that
$mB_{X_0}$ is Cartier on $U$. 
Let $\mathcal{I}$ be the ideal sheaf on $X_0$
defined by the Weil divisor $mB_{X_0}$. 
By Lemma~\ref{lem:principalization-on-sing-var}, we know that there exists a sequence of blow-ups 
$X_1\rightarrow X_0$
along subvarieties of codimension at least two in $X\setminus U$ which is a principalization for $mB_{X_0}$. 
As $\pi_1\colon X_1\rightarrow X_0$ is an isomorphism along $U$, we have an embedding $U\hookrightarrow X_1$.
Consider $E_1$ the reduced divisor supported on $X_1\setminus U$.
Let $B_{X_1}$ be the closure of $B_U$ on $X_1$. 
As $X_1\rightarrow X_0$ is a principalization of $mB_{X_0}$, we know there exists a Weil divisor $S$
whose support is contained in $\supp(E_1)$ 
such that $mB_{X_1}+S$ is Cartier on $X_1$. 
On the other hand, as $X_1\rightarrow X_0$ is a sequence of blow-ups of subvarieties of codimension at least two, there exists an effective divisor $G$ on $X_1$ such that $-G$ is ample over $X_0$.
By construction, the image of $G$ on $X_0$ is contained in $X_0\setminus U$. 
Therefore, for some small rational number $\epsilon$, the divisor $\pi_1^*A_{X_0}-\epsilon G$ is an ample effective divisor on $X_1$ and its support equals $X_1\setminus U$. 
We let $A_{X_1}:=\ell(\pi_1^*A_{X_0}-\epsilon G)$ where $\ell$ is large enough so that $A_{X_1}$ is a very ample divisor. We replace $(X_0,B_{X_0},A_{X_0})$
with $(X_1,B_{X_1},A_{X_1})$. 
By doing so, the following conditions are satisfied: 
\begin{enumerate}
\item $X_0$ is a normal projective variety, 
\item there is an effective $\qq$-Cartier divisor $B_{X_0}$ on $X_0$ such that $B_{X_0}|_{U}:=B_U$, and 
\item $X_0\setminus U$ fully supports a very ample effective divisor $A_{X_0}$. 
\end{enumerate} 
Let $p\colon Y\rightarrow X_0$ be a log resolution of $(X_0,B_{X_0}+A_{X_0})$. 
We know that $X$ is a rational variety that contains $U\supset \mathbb{G}_m^n$. 
We may assume that there is a projective birational morphism $q \colon Y\rightarrow \pp^n$ which is an isomorphism on $\mathbb{G}^n_m$. Further, we can assume that there is an effective $p$-exceptional divisor $G_Y \geq 0$ which is anti-ample over $X$. 
Let $E_Y$ be the reduced exceptional divisor of $p$ 
plus the strict transform of the reduced exceptional divisor supported on $X_0\setminus U$. 
Let $B_Y$ be the strict transform on $Y$ of the closure of $B_U$ in $X_0$.
By construction, the pair 
$(Y,B_Y+E_Y)$ is a $\qq$-factorial dlt pair.
Consider the log pair 
\begin{equation}\label{eq:log-pair-p^n}
(\pp^n,H_0+\dots+H_n) 
\end{equation} 
where the $H_i$'s are the poles of $\Omega$ on $\pp^n$.
There is a $\qq$-linear equivalence 
\[
K_Y+B_Y+E_Y \sim_\qq F_Y, 
\]
where $F_Y$ is an effective divisor supported along all the prime $q$-exceptional divisors that have positive log discrepancy with respect to $(\pp^n,H_0+\dots+H_n)$.
Indeed, this follows from the fact that
the pair~\eqref{eq:log-pair-p^n} has log canonical singularities and it is crepant birational equivalent to $(U,B_U)$. 
We run a $(K_Y+B_Y+E_Y)$-MMP over $X_0$ with scaling of an ample divisor. 
Note that every prime component of $F_Y$ whose center on $X$ intersects $U$ is a degenerate divisor over $X_0$. 
Therefore, after finitely many steps, this MMP 
terminates over $U$ (see, e.g.,~\cite[Lemma 2.9]{Lai11}). 
Let $Y_1$ be the model over $X_0$ on which this MMP terminates over $U$ 
and let $\phi_1\colon Y \dashrightarrow Y_1$ be the induced projective birational contraction over $X_0$. 
We denote by $\pi \colon Y\rightarrow X_0$ the induced projective birational morphism.
Let $B_{Y_1}$ and $E_{Y_1}$ be the push-forward of $B_Y$ and $E_Y$ on $Y_1$, respectively. 
Then, the pair $(Y_1,B_{Y_1}+E_{Y_1})$ is a $\qq$-factorial dlt pair.
Let $V_1\subset Y_1$ be the pre-image of $U$ in $Y_1$ and let $(V_1,B_{V_1}+E_{V_1})$ be the restriction of the aforementioned pair to $V_1$. 
Then, the morphism $(V_1,B_{V_1}+E_{V_1})\rightarrow (U,B_U)$ is a $\qq$-factorial dlt modification. 
We let $\pi_{V_1}$ be the restriction of $\pi_1$ to $V_1$.
Let $F_{Y_1}$ be the push-forward of $F_Y$ on $Y_1$. 
Then, by construction, we have that 
\[
K_{Y_1}+B_{Y_1}+E_{Y_1}\sim_\qq F_{Y_1} 
\]
and the effective divisor $F_{Y_1}$ satisfies 
${\rm Bs}_{-}(F_{Y_1})\subseteq {\rm supp}(F_{Y_1})$. 
Further, the support of $F_{Y_1}$ is in the complement of $V_1$. 
By~\cite[Lemma 2.9]{Lai11}, we may run an MMP with scaling of an ample divisor for $(Y_1,B_{Y_1}+E_{Y_1})$
which terminates after contracting all the components of $E_{Y_2}$. 
Let $\phi_2\colon Y_1 \dashrightarrow Y_2$ be the induced MMP. 
Let $B_{Y_2}$ and $E_{Y_2}$ be the push-forward to $Y_2$ of the divisors $B_{Y_1}$ and $E_{Y_1}$, respectively.
Then, the pair $(Y_2,B_{Y_2}+E_{Y_2})$ is a $\qq$-factorial dlt pair. 
Furthermore, by construction, the pair 
$(Y_2,B_{Y_2}+E_{Y_2})$ is crepant birational equivalent to the pair~\eqref{eq:log-pair-p^n} 
and so it is a log Calabi--Yau pair. 

We argue that $Y_2$ is a Fano type variety. 
By conditions (2) and (3) above, we know that there exists an ample effective $\qq$-Cartier divisor $H_{X_0}$ whose support equals such of $\supp(B_{X_0}+A_{X_0})$. 
Thus, the divisor $H_{X_0}$ contains every log canonical center of $(U,B_U)$. 
Let $H_Y:=p^*H_{X_0} - \epsilon G_Y$ for some small rational number $\epsilon$. 
Then, the divisor $H_Y$ is ample on $Y$.
Note that the divisor $H_Y$ is not effective.
However, all the negative coefficients of $H_Y$ happen along prime divisors that are not
log canonical places of the pair~\eqref{eq:log-pair-p^n}. 
By the negativity lemma, the indeterminacy locus of
$\phi_1^{-1}\circ \phi_2^{-1}\colon Y_2 \dashrightarrow Y_1$ contains no log canonical center of $(Y_2,B_{Y_2}+E_{Y_2})$. 
Indeed, if the indeterminancy locus of $\phi_1^{-1}\circ \phi_2^{-1}$ contained the center $c_{Y_2}(E)$ where $E$ is a log canonical place of $(Y_2,B_{Y_2}+E_{Y_2})$, then the monotonicity of
log discrepancies would imply
\[
a_E(Y_1,B_{Y_1}+E_{Y_1}) < a_E(Y_2,B_{Y_2}+E_{Y_2}) = 0, 
\]
see e.g.,~\cite[Lemma 2.3]{Mor25}. Thus, we would conclude that $(Y_1,B_{Y_1}+E_{Y_1})$ is not lc, leading to a contradiction. 
Hence, 
if $0\leq S_Y\sim_\qq H_Y$ is a general effective divisor
in the $\qq$-linear system $|H_Y|_\qq$
and $S_{Y_2}$ is its image in $Y_2$, then the pair
$(Y_2,B_{Y_2}+E_{Y_2}+\delta S_{Y_2})$ is log canonical for $\delta>0$ small enough.
Indeed, in this case, the divisor $S_{Y_2}$
contains no log canonical center of the $\qq$-factorial dlt pair $(Y_2,B_{Y_2}+E_{Y_2})$. 
Let $H_{Y_2}$ be the image of $H_Y$ on $Y_2$.
As $Y_2\dashrightarrow \pp^n$ only extract log canonical places of the pair~\eqref{eq:log-pair-p^n}, 
we conclude that $H_{Y_2}$ is an effective divisor
whose support equals such of $B_{Y_2}+E_{Y_2}$. 
Hence, for $\delta>0$ small enough, 
the log Calabi--Yau pair
\[
\left( Y_2, B_{Y_2}+E_{Y_2}-\delta H_{Y_2}+\delta S_{Y_2} \right) 
\]
is klt and the divisor $S_{Y_2}$ is big.
Hence, $Y_2$ is of Fano type. 

Recall that the MMP $Y_1\dashrightarrow Y_2$
is an isomorphism over $V_1$, so we may set
$V_2\simeq V_1$ its image on $Y$.  
Let $P_V$ be a prime divisor on $V_2$ 
which is contracted by the projective birational morphism $V_2\rightarrow U$. 
Let $P$ be the closure of $P_V$ in $Y_2$. 
The divisor $P$ is covered by curves that intersect $P$ negatively. Thus, we have ${\rm Bs}_{-}(P)\supseteq P$. 
Hence, we may run a $P$-MMP that terminates after contracting $P$.
Note that we are allowed to run this MMP as $Y_2$ is of Fano type and hence a Mori dream space by Lemma~\ref{lem:FT-are-MDS}.
Proceeding inductively, we obtain a projective birational map $\phi_3\colon Y_2\dashrightarrow Y_3$ that contracts the closure of all the prime divisors on $V_2$ which are contracted on $U$.
By construction, the variety $Y_3$ is a $\qq$-factorial Fano type variety (see Lemma~\ref{lem:FT-bir-cont}).
As usual, we let $B_{Y_3}$ and $E_{Y_3}$ the push-forward to $Y_3$ of $B_{Y_2}$ and $E_{Y_2}$, respectively. We denote by $V_3$ the complement
of $E_{Y_3}$ in $Y_3$. 
We obtain a commutative diagram as follows:
\[
\xymatrix{ 
Y \ar[r]^-{q}\ar[d]_-{p}\ar@{-->}[rd]^-{\phi_1} & \pp^n & &  \\
X_0 & Y_1 \ar[l]^-{\pi} \ar@{-->}[r]^-{\phi_2} & Y_2 \ar@{-->}[r]^-{\phi_3} & Y_3 \\
U \ar@{^{(}->}[u]  & V_1 \ar[l]_-{\pi_{V_1}} \ar@{^{(}->}[u] \ar@{-->}[r] & V_2 \ar@{^{(}->}[u]\ar@{-->}[r] & V_3\ar@{^{(}->}[u] \\
}
\]
By construction, the variety $V_3$ is small birational to the variety $U$.
Let $A_Y$ be the pull-back of $A_X$ to $Y$.
Then, the divisor $A_Y$ is an effective big divisor
which has the same support as $E_Y$.
Let $A_{Y_3}$ be the image of $A_Y$ on $Y_3$.
Then, the divisor $A_{Y_3}$ is an effective big divisor
that has the same support as $E_{Y_3}$.
As $Y_3$ is a Mori dream space, we can take the ample model of $A_{Y_3}$. Let $Y_3\dashrightarrow Y_4$ be the ample model of $A_{Y_3}$. 
Let $A_{Y_4}$ be the image of $A_{Y_3}$ in $Y_4$ which is an ample divisor.
Therefore, its complement in $Y_4$, call it $V_4$, is an affine variety. 
Again, the birational map $V_3\dashrightarrow V_4$ is a small birational map.
Thus, the induced map $U\dashrightarrow V_4$ is a small birational map between affine varieties and so it is an isomorphism. 
We set $X:=Y_4$, 
$B:=B_{Y_4}$, and 
$A:=A_{Y_4}$. 
By construction, the pair $(X,B)$ is of cluster type.
We have an isomorphism $U\simeq V_4 \simeq X\setminus A$.
So, taking the structure sheaf, we get 
\[
R \simeq \mathcal{O}(U) \simeq H^0(X\setminus A, \mathcal{O}_X). 
\]\par
Finally, we show that (3) implies (1). Since $(X,B)$ is cluster type, there exists a birational map $j_0\colon \mathbb{G}_m^n\dashrightarrow X\setminus B$ that is an open embedding in codimension one. Let $W_0\subset \mathbb{G}_m^n$ be a big open on which $j_0$ restricts to be an open embedding, and let $E_1,\hdots, E_n\subset X\setminus B$ be those prime divisors on $X\setminus B$ that are not contained in $j_0(W_0).$ For each $1\leq i \leq n,$ ~\cite[Lemma 16.2.8]{corti2023cluster} provides us with a mutation $\mu_i\colon \mathbb{G}_m^n\dashrightarrow \mathbb{G}_m^n$ satisfying the following conditions:
\begin{enumerate}
    \item[(i)] $j_i:=j_0\circ \mu_i\colon \mathbb{G}_m\dashrightarrow X\setminus B$ is an open embedding in codimension one, and
    \item[(ii)] if $W_i\subset \mathbb{G}_m^n$ is a big open on which $j_i$ restricts to be an open embedding, then $j_i(W_i)$ meets $E_i$.
\end{enumerate}
For each $0\leq i \leq n,$ let $\widetilde{j}_i\colon V_i\rightarrow X\setminus A$ be the extension of $j_i$ as in Lemma ~\ref{lem:affine-partial-compactifications}. Choose big opens $U_i\subset V_i$ on which the $\widetilde{j}_i$ restrict to be open embeddings, and write $R_i=\Gamma(\widetilde{j}_i(U_i),\mathcal{O}_X).$ By construction, the sets $\widetilde{j}_0(U_0),\hdots,\widetilde{j}_n(U_n)$ provide an open cover of a big open subset of $X\setminus A$. Thus, $R=R_0\cap \hdots \cap R_n.$ We claim that this presentation satisfies the conditions of Definition ~\ref{def:MSA}. Condition (1) of Definition ~\ref{def:MSA} holds with $\iota_i\colon \Gamma(V_i,\mathcal{O}_{V_i})\hookrightarrow {\rm Frac}(R)$ the embedding induced by the birational moprhism $\widetilde{j}_i.$ The birational map $\widetilde{j}_0^{-1}\circ \widetilde{j}_i=\mu_i$ restricts to a mutation $\mathbb{G}_m^n\dashrightarrow\mathbb{G}_m^n$ by construction. It follows from this and Lemma ~\ref{lem:affine-partial-compactifications} that it is a mutation $V_0\dashrightarrow V_i$ in the sense of Definition ~\ref{def:mutation}. Thus, condition (2) of Definition ~\ref{def:MSA} holds. Condition (3) of Definition ~\ref{def:MSA} follows from the fact that each $\widetilde{j}_i$ is an open embedding in codimension one.
\end{proof} 

\begin{proof}[Proof of Corollary~\ref{introcor:msa-Fano-comp}]
Let $R$ be a finitely generated commutative $\kk$-algebra.
Assume that $R$ is a mutation semigroup algebra. 
Furthermore, assume that $U:={\rm Spec}(R)$ has klt singularities. 
By the construction given in the proof of Theorem~\ref{thm:geom-msa}, we may assume, moreover, that the variety is Fano type.
By Proposition~\ref{prop:equiv-FT-logF}, $X$ is log Fano. This proves (1). 
If $U$ is $\qq$-factorial, we may replace $X$ with a small $\qq$-factorialization $X'$ which is an isomorphism over $U$. Replacing $X$ with $X'$ does not change the Fano type property (see, e.g.,~\cite[Lemma 2.17.(3)]{Mor24a}). Thus, Proposition~\ref{prop:equiv-FT-logF} shows that $X'$ is log Fano. This proves the last statement of the Corollary for (1). 

Now, assume that the volume form $\Omega_U$ of $U$ has no poles on $U$. Therefore, we have $U\subset X\setminus B$. Thus, the ample divisor $A$ has the same support as the divisor $B$.
However, a priori, the divisor $B$ may not be ample, or not even $\qq$-Cartier. 
Let $X'\rightarrow X$ be a small $\qq$-factorialization of $X$.
Note that $X'\rightarrow X$ is a small morphism
which is a Fano type morphism.
Let $B'$ be the strict transform of $B$ in $X'$
and $U'$ be the preimage of $U$ in $X'$.
Let $X''\rightarrow X$ be the ample model of $B'$ over $X$. As $B'$ is disjoint from $U'$ and every curve contracted by $U'\rightarrow U$ is $B'$-trivial, we conclude that there is an embedding $U\hookrightarrow X''\setminus B''$ where $(X'',B'')$ is a cluster type pair and $-K_{X''}\sim B''$. 
Note that $X''$ is still a Fano type variety.
Let $Y$ be the ample model of $B''$ and let $B_Y$
be the image of $B''$ in $Y$.
Therefore, we have $U\hookrightarrow Y\setminus B_Y$
and $B_Y$ is an ample divisor.
Furthermore, the pair $(Y,B_Y)$ is of cluster type
and $-K_Y \sim B_Y$ is an ample divisor.
Furthermore, by construction, the pair $(Y,B_Y)$ is log canonical and every log canonical center is contained in the support of $B_Y$. Therefore, 
the variety $Y$ has klt singularities and $-K_Y$ is ample. Thus, $Y$ is a klt Fano compactification of $U$.
This proves (2).

If $U$ is $\qq$-factorial, then we can take $X''=X'$ in the previous argument. Indeed, we still have an embedding $U\hookrightarrow X'\setminus B'$.
\end{proof}

The proof of Theorem~\ref{thm:geom-msa} works in the equivariant setting with respect to a quasi-torus action. Indeed, all the employed constructions: resolution of singularities, minimal model program, and Gruson-Raynaud flattening, work in the equivariant setting (see, e.g.,~\cite[Proposition 3.9.1]{Kol07},~\cite[Section 3.1]{Pro21},~\cite[Theorem A]{Ryd25}).
Thus, we obtain the following corollary.

\begin{corollary}\label{cor:geom-msa-equiv}
Let $\kk$ be an algebraically closed field of characteristic zero, let
$K$ be a finitely generated abelian group,
and let 
$R$ be a $K$-graded commutative ring 
which is finitely generated over $\kk$.
Let $\mathbb{T}={\rm Spec}(\kk[K])$ be the associated quasi-torus.
Assume that $U:={\rm Spec}(R)$ has klt singularities. 
Then, the two following conditions are equivalent:
\begin{enumerate}
\item The ring $R$ is a $K$-graded mutation semigroup algebra over $\kk$, and 
\item there is a $\mathbb{T}$-equivariant open embedding in codimension one $\mathbb{G}_m^n\dashrightarrow U$ such that the volume form $\Omega$ of $\mathbb{G}_m^n$ vanishes nowhere on $U$. 
\end{enumerate} 
\end{corollary}

Now, we turn to prove the last two corollaries of Theorem~\ref{thm:geom-msa}.

\begin{proof}[Proof of Corollary~\ref{introcor:nice-comp-FT}]
Let $R$ be a mutation semigroup algebra and $U={\rm Spec}(R)$. 
Let $\Omega_U$ be the volume form on $U$.
Let $U\hookrightarrow X$ be a compactification for which $\Omega_U$ has a pole at every prime component of $X\setminus U$.
Since $\Omega_U$ has poles at all the components of $X\setminus U$, there exists $B\geq 0$ with $K_X+B\sim 0$ and $X\setminus B=U$.
By Theorem~\ref{thm:geom-msa}, we know that there exists an embedding $U\hookrightarrow X_0$
where $X_0$ is projective, $(X_0,B_0)$ is a cluster type pair, and there exists $A_0\leq B_0$ ample with $X_0\setminus A_0=U_0$.
In particular, $X_0$ is a Fano type variety
and the pair $(X_0,B_0)$ is crepant birational equivalent to $(X,B)$.
Therefore, there exists a projective birational morphism $Y\rightarrow X_0$, that only extracts log canonical places of $(X_0,B_0)$ and $Y\dashrightarrow X$ is a birational contraction 
(see, e.g.,~\cite[Theorem 1]{Mor19}).
By~\cite[Lemma 2.17.(3)]{Mor24a}, we conclude that $Y$ is Fano type.
By~\cite[Lemma 2.17.(1)]{Mor24a}, we conclude that $X$ is of Fano type.
Therefore, $X$ is log Fano and a Mori dream space.
\end{proof}

\begin{proof}[Proof of Corollary~\ref{introcor:comp-locally-acyclic}]
Let $R$ be a locally acyclic cluster algebra.
Let $U={\rm Spec}(R)$ be its spectrum.
By~\cite{BMRS15}, we know that $U$ has canonical singularities. 
By Theorem~\ref{thm:geom-msa}, we know there exists a klt Fano compactification $U\hookrightarrow X$. Furthermore, there exists a cluster type pair $(X,B)$ such that $U=X\setminus B$. 
Let $p\colon Y\rightarrow X$ be a $\qq$-factorial dlt modification of $(X,B)$. 
By~\cite[Lemma 2.17.(1)]{Mor24a}, we know that $Y$ is a log Fano variety. 
As $U\simeq X\setminus B$ and $U$ has canonical singularities, we conclude that $p$ is an isomorphism along $U$.
Write $p^*(K_X+B)=K_Y+B_Y$ so $K_Y+B_Y\sim 0$ and $B_Y$ is effective. 
We have $U\simeq Y\setminus B_Y$ and $Y$ is a log Fano variety, indeed, $B_Y$ fully supports an ample divisor.
It suffices to show that $Y$ has canonical singularities. 
Let $E$ be a divisorial valuation over $Y$.
Assume $a_E(Y)<1$. Then, we have
$a_E(Y,B_Y)=0$ as $a_E(Y,B_Y)<1$ must be a non-negative integer. 
Therefore, $c_Y(E)$ lies on the smooth locus of $Y$, from where we deduce that $a_E(Y)\geq 1$, leading to a contradiction.
We deduce that $a_E(Y)\geq 1$ for every divisorial valuation $E$ over $Y$.
Thus, $Y$ has canonical singularities. 
Therefore, $Y$ is a canonical log Fano compactification of $U$.
\end{proof} 

\section{Mutation semigroup algebras as Cox rings}
\label{sec:MA-vs-CR}
In this section, we show that Cox rings of 
cluster type varieties are graded MSA's. 
Furthermore, this statement characterizes 
cluster type varieties among Fano type varieties. 

\begin{proof}[Proof of Theorem~\ref{introthm:cluster-type-Cox-rings}]
Let $X$ be a $\qq$-factorial Fano variety.
Assume that ${\rm Cox}(X)$ is a ${\rm Cl}(X)$-graded
mutation semigroup algebra.
We let $\rho$ be the Picard rank on $X$.
Let $\bar{X}$ be the Cox space of $X$, i.e., 
the spectrum of the Cox ring of $X$. 
Let $\mathbb{T}$ be the class group quasi-torus acting on $\bar{X}$.
By~\cite[Theorem 1]{Bra19}, we know that $\bar{X}$ is a Gorenstein variety.
By~\cite[Theorem 1.1]{GOST15}, we know that $\bar{X}$ is log
terminal.
Therefore, the variety $\bar{X}$ is an affine variety 
with canonical singularities.
By Theorem~\ref{thm:geom-msa}, we know that there exists an open embedding 
$\mathbb{G}_m^{n+\rho}\hookrightarrow \bar{X}$ such
that the volume form $\Omega_{\bar{X}}$ 
of the algebraic torus 
has no zeros on $\bar{X}$.
As ${\rm Cox}(X)$ is a ${\rm Cl}(X)$-graded
mutation semigroup algebra, 
the torus embedding $\mathbb{G}_m^{n+\rho}\hookrightarrow \bar{X}$ is 
$\mathbb{T}$-equivariant. 
In particular, the algebraic quasi-torus $\mathbb{T}$ acts linearly on $\mathbb{G}_m^{n+\rho}$. 
Let $Z\subset \bar{X}$ be the irrelevant locus
of the Cox space.
Thus, we have a good quotient
$\pi\colon \bar{X}\setminus Z \rightarrow X$ for the $\mathbb{T}$-action. 
By construction, we know that $Z$ has codimension at least two and it is $\mathbb{T}$-invariant.
In particular, we conclude that 
\[
\mathbb{G}_m^{n+\rho} \setminus Z \simeq 
\mathbb{T} \times \mathbb{G}_m^{n}\setminus Z_0,   
\]
where $Z_0\subset \mathbb{G}_m^n$ is a closed subset of codimension at least two. 
Thus, the variety $X$ admits an open subset isomorphic to 
\[
(\mathbb{G}_m^{n+\rho}\setminus Z)/\!/ \mathbb{T} 
\simeq \mathbb{G}_m^n \setminus Z_0.
\]
Thus, we conclude that there is an embedding 
in codimension one 
$\mathbb{G}_m^n \dashrightarrow X$.
Let $\Omega_X$ be the volume form of this algebraic torus. It suffices to show that $\Omega_X$ has no zeros on $X$. 
In order to do so, it suffices to show that there exists a log Calabi--Yau pair structure $(X,B)$ 
such that $\mathbb{G}_m^n\dashrightarrow X\setminus B$
is an embedding in codimension one. 
Let $Y$ be the normalized Chow quotient for
the action of $\mathbb{T}$ on $\bar{X}$. 
Then, we have a projective birational morphism
$\phi\colon Y\rightarrow X$  (see, e.g.,~\cite[\S 5.2]{AW11}). 
There exists a $\mathbb{T}$-equivariant 
projective birational morphism $r\colon \widetilde{X}\rightarrow \bar{X}$ and a good quotient $\pi_0\colon \widetilde{X}\rightarrow Y$ for the $\mathbb{T}$-action (see, e.g.,~\cite[Theorem 1]{AH06}). 
We obtain a commutative diagram as follows:
\[
\xymatrix{
\widetilde{X} \ar[d]_-{\pi_0} \ar[r]^-{r}  & \bar{X} \ar@{-->}[d]^-{\pi} \\ 
Y\ar[r]^-{\phi} & X.
}
\]
Let $(\bar{X},\bar{B})$ be the Calabi--Yau pair on $\bar{X}$ induced by the poles of the volume form $\Omega_{\bar{X}}$. 
If we write 
\[
K_{\widetilde{X}}+B_{\widetilde{X}} =
r^*(K_{\bar{X}}+\bar{B}), 
\]
then all the prime $\mathbb{T}$-invariant divisors
of $\widetilde{X}$ that are horizontal over $Y$ are components of $B_{\widetilde{X}}^{=1}$. 
Indeed, all the prime $\mathbb{T}$-invariant divisors
of $\widetilde{X}$ that are horizontal over $Y$
are poles of $\Omega_{\bar{X}}$. 
Thus, the sub-log Calabi--Yau pair $(\widetilde{X},B_{\widetilde{X}})$ has a log canonical center $Y_0$ that maps birationally to $Y$.
Let $(Y_{0},B_{Y_0})$ be the sub-log Calabi--Yau pair
obtained by adjunction of $(\widetilde{X},B_{\widetilde{X}})$ to $Y_0$.
We let $B_Y$ be the push-forward of $B_{Y_0}$ to $Y$.
By the negativity lemma, the pair
$(Y,B_Y)$ is a sub-log Calabi--Yau pair. 
Let $B_X$ be the push-forward of $B_Y$ on $X$.
Then, the pair $(X,B_X)$ is a sub-log Calabi--Yau pair of index one. 
By construction, the divisor $B_X$ is in the complement of $\mathbb{G}_m^n$ in $X$.
Indeed, this follows from the fact that $\pi\colon \bar{X}\dashrightarrow X$ is an \'etale principal $\mathbb{T}$-bundle over $X^{\rm sm}$ (see~\cite[Proposition 1.6.1.6]{ADHL15}).
We argue that $B_X$ is an effective divisor.
Indeed, the map $\pi$ is surjective, therefore, 
for every prime divisor $P$ of $X$ 
there is a $\mathbb{T}$-invariant prime divisor $Q$ on $\widetilde{X}$ mapping onto $P$ for which
${\rm coeff}_{\rm Q}(B_{\widetilde{X}})\geq 0$.
By construction, the coefficient of $B_X$ at $P$
must be at least
\[
\frac{1-m}{m} + \frac{ {\rm coeff}_{\rm Q}(B_{\widetilde{X}})}{m} \geq 0
\]
where $m$ is the multiplicity of $\pi_0^*\pi^*P$ at $Q$. 
The previous statement follows from the adjunction formula, which determines the coefficients of $B_{Y_0}$, (see, e.g.,~\cite[Corollary 3.10]{Sho92} and~\cite[Theorem 3.21]{PS11}).
Thus, we conclude that $B_X$ is an effective divisor.
Hence, the pair $(X,B_X)$ is log Calabi--Yau and
$\mathbb{G}_m^n\dashrightarrow X\setminus B_X$ is a embedding in codimension one.
We conclude that $X$ is a cluster type variety by Lemma~\ref{lem:cluster-type-via-tori}.

Now, assume that $X$ is a cluster type variety.
Let $(X,B)$ be a cluster type pair structure. 
Therefore, we have that $(X,B)$ is log Calabi--Yau
and there is an open embedding 
$\mathbb{G}^n_m\hookrightarrow X\setminus B$ (see, e.g.,~\cite[Theorem 1.3.(4)]{EFM24}).
Let $\Omega_X$ be the volume form of the algebraic torus. 
Let $p_X\colon \hat{X}\rightarrow X$ be the relative spectrum
of the Cox sheaf on $X$. 
Then $p_X$ is an \'etale $\mathbb{T}$-principal 
bundle over the smooth locus of $X$ (see~\cite[Proposition 1.6.1.6.(i)]{ADHL15}).
In particular, we can find an embedding 
$\mathbb{G}_m^{n+\rho}\hookrightarrow \bar{X}$ (see~\cite[Construction 1.6.1.5]{ADHL15} and~\cite[Theorem 3]{BM21}).
Let $\Omega_{\bar{X}}$ be the associated volume form on $\bar{X}$.
Let $Y$ be the normalized Chow quotient
for the $\mathbb{T}$-action on $\bar{X}$.
Then, there is a $\mathbb{T}$-equivariant 
projective birational morphism
$r\colon \widetilde{X}\rightarrow \bar{X}$ 
and a good quotient $\pi_0\colon \widetilde{X}\rightarrow Y$ for the $\mathbb{T}$-action.
Write $K_X+B={\rm div}(f)$ with $f\in \mathbb{K}(X)$. 
Let $K_Y+B_Y=\pi^*(K_X+B)$. 
A canonical divisor $K_Y$ on $Y$ 
determines a canonical divisor $K_{\widetilde{X}}$ on $\widetilde{X}$ (see e.g.,~\cite[Theorem 3.21]{PS11}).
Let $E$ be the reduced sum of all the $\mathbb{T}$-invariant prime divisors of $\widetilde{X}$ which are horizontal over $Y$.
Let $B_{\widetilde{X}}$ be the unique Weil divisor for which
\[
K_{\widetilde{X}}+B_{\widetilde{X}}+E = {\rm div}(f),
\]
where $f$ is considered as a $\mathbb{T}$-invariant rational function on $\widetilde{X}$
via the isomorphism $\mathbb{K}(\widetilde{X})^{\mathbb{T}}\simeq \mathbb{K}(X)[M]$ (see, e.g.,~\cite[Section 11.2]{AIPSV12}).
The volume form $\Omega_{\bar{X}}$ has a pole along 
$\pi_0^*P$ if and only if $\Omega_X$ has a pole along $P$. 
By construction, every prime component of $B_{\widetilde{X}}$ with non-positive coefficient 
gets contracted in $\bar{X}$. 
Indeed, every prime component of $B_Y$ with non-positive coefficient gets contracted in $X$. 
Therefore, if $\bar{B}$ is the push-forward of
$B_{\widetilde{X}}$ in $\bar{X}$,
then $(\bar{X},\bar{B})$ is a log Calabi--Yau pair
and $\mathbb{G}_m^{n+r}\hookrightarrow \bar{X}\setminus \bar{B}$.
Furthermore, the boundary divisor $\bar{B}$ is given by the set of poles of $\Omega_{\bar{X}}$. 
By~\cite[Theorem 1]{Bra19}, we know that $\bar{X}$ is Gorenstein,
and by~\cite[Theorem 1.1]{GOST15}, we know that $\bar{X}$ has klt type singularities.
Thus, we conclude that $\bar{X}$ is an affine variety with klt singularities.
Hence, by Corollary~\ref{cor:geom-msa-equiv}, we conclude that the ring  ${\rm Cox}(X)$ is a ${\rm Cl}(X)$-graded volume preserving algebra.
\end{proof} 

\begin{proof}[Proof of Theorem~\ref{introthm:Cox-space-covered-algebraic-torus}]
Let $(X,B)$ be a cluster type structure of $X$. 
Let $P$ be a prime divisor on $X$. 
Let $\pi\colon \bar{X}\dashrightarrow X$ be
the rational quotient map from its Cox space.
It suffices to show that there is an embedding 
$j\colon \mathbb{G}_m^n\hookrightarrow \bar{X}$ such
that $j(\mathbb{G}_m^n)$ contains the generic point of $\pi^{-1}(P)$. 
By~\cite[Theorem 1.3]{EFM24}, we know that $X\setminus B$ 
is covered by copies of $\mathbb{G}_m^n$ 
up to a closed subset $Z\subset X\setminus B$ 
of codimension at least two. 
Assume that the generic point of $P$ 
is contained in $X\setminus B$.
Therefore, it is also contained in $X\setminus (B\cup Z)$.
So there is an embedding 
$j_0\colon \mathbb{G}_m^n \hookrightarrow X\setminus (B\cup Z)$ such that $j_0(\mathbb{G}_m^n)$ contains 
the generic point of $P$. 
Let $\pi_0\colon \bar{X}\rightarrow X$ be the relative spectrum 
over $X$ of the Cox sheaf of $X$. 
The morphism $\pi_0$ is an \'etale principal $\mathbb{T}_{{\rm Cl}(X)}$-torsor over $X^{\rm sm}$ (see, e.g.,~\cite[Proposition 1.6.1.6]{ADHL15}).
By~\cite[Theorem 3]{BM24}, we know that $\pi_0\colon \pi_0^{-1}(j_0(\mathbb{G}_m^n))\rightarrow j_0(\mathbb{G}_m^n)$ is the composition of an \'etale finite cover and a trivial $\mathbb{G}_m^\rho$-bundle, so we conclude that there is an embedding $\mathbb{G}_m^{n+\rho}\hookrightarrow \hat{X}$.
By~\cite[Construction 1.6.1.5]{ADHL15}, we know that there is an embedding $\hat{X}\hookrightarrow \bar{X}$. 
Thus, we conclude that there is an embedding 
$\mathbb{G}_m^{n+\rho}$ that contains the generic point of
$\pi^{-1}(P)$. 

From now on, we may assume that $P$ is a prime component of $B$. 
By Lemma~\ref{lem:toric-model-cluster-type}, there exists a crepant birational map $(T,B_T)\dashrightarrow (X,B)$ satisfying the following conditions:
\begin{enumerate}
\item the pair $(T,B_T)$ is toric,
\item the center $P_T$ of $P$ in $T$ is a prime divisor,
\item for every prime component $Q\subset B$ the center
of $Q$ in $T$ is a prime divisor, and 
\item for every prime divisor $Q\subset X$ with $Q\not\subset B$ either 
\begin{enumerate}
\item[(a)] the center of $Q$ in $T$ is a prime divisor, or 
\item[(b)] the center of $Q$ in $T$ is a codimension two subset which is not contained in $P_T$. 
\end{enumerate} 
\end{enumerate} 
By Lemma~\ref{lem:torus-in-toric}, there is an embedding 
$j_T\colon \mathbb{G}_m^n \hookrightarrow T$ such that 
$j_T(\mathbb{G}_m^n)$ contains the generic point of $P_T$
and $T\setminus j_T(\mathbb{G}_m^n)$ contains every prime torus invariant divisor of $T$ except for $P_T$.
Let $(Y,B_Y)\rightarrow (T,B_T)$ be the crepant birational 
projective morphism obtained by extracting 
every $Q\subset X$ with $Q\not \subset B$ (see~\cite[Theorem 1]{Mor19}). 
Then, we have two birational maps 
\[
\xymatrix{
(Y,B_Y) \ar[d]_-{\psi}\ar@{-->}[rd]^-{\phi} & \\ 
(T,B_T) & (X,B), 
}
\]
where $\phi\colon (Y,B_Y)\dashrightarrow (X,B)$ is a crepant birational contraction. 
By passing to a small $\qq$-factorialization,
we may assume that $Y$ is $\qq$-factorial.
Let $P_Y$ be the strict transform of $P_T$ in $Y$.
By condition (4).(b), we have an embedding $j_Y\colon \mathbb{G}_m^n$ such that $j_Y(\mathbb{G}_m^n)$ contains the generic point of $P_Y$ and $Y\setminus j_Y(\mathbb{G}_m^n)$ contains every component of $B_Y$ except for $P_Y$.
The birational map $\phi$ is a birational contraction
that only extracts log canonical places of $(X,B)$.
Therefore, the variety $Y$ is of Fano type by Lemma~\ref{lem:FT-higher-model}.
Let $S_Y$ be the reduced exceptional divisor of $\phi$. 
The divisor $S_Y$ is degenerate over $X$ so by the negativity lemma we have ${\rm supp}(S_Y) \subseteq {\rm Bs}_{-}(S_Y)$.
As $Y$ is a Mori dream space, by Lemma~\ref{lem:FT-are-MDS}, we can run a $S_Y$-MMP that terminates after contracting all the components of $S_Y$.
Let $\phi'\colon Y \dashrightarrow X'$ be the outcome of this MMP. Let $B'$ be the push-forward of $B_Y$ in $X'$.
Let $P'$ be the push-forward of $P_Y$ in $X'$.
Thus, we obtain a commutative diagram:
\[
\xymatrix{
(Y,B_Y) \ar[d]^-{\psi} \ar@{-->}[r]^-{\phi'}
\ar@{-->}[rd]^-{\phi}
& (X',B') \ar@{-->}[d]^-{\rho} \\
(T,B_T) & (X,B) 
}
\]
where both $\phi$ and $\phi'$ are birational contractions 
and $\rho$ is a small crepant birational map.
We know that $S_Y$ is disjoint from $j_Y(\mathbb{G}_m^n)$. 
so  there is an embedding $j'\colon \mathbb{G}_m^n\hookrightarrow X'$
such that $j'(\mathbb{G}_m^n)$ contains the generic point of
$P'$ and $X'\setminus j'(\mathbb{G}_m^n)$ contains all the components of $B'$ except for $P'$. 
As $\rho$ is a small birational map between $\qq$-factorial  Mori dream spaces, we have two rational maps 
from the Cox space
\[
\xymatrix{ 
& \bar{X} \ar@{-->}[ld]_-{\pi}\ar@{-->}[rd]^-{\pi'} & \\
X & & X'
}
\]
corresponding to different GIT quotients. 
Arguing as in the first paragraph\footnote{the argument in the first paragraph applies verbatim after replacing $X'$ with $X$.}, we conclude that 
there exists an embedding $\mathbb{G}_m^{n+\rho}\hookrightarrow \bar{X}$ that contains the generic point of ${\pi'}^{-1}(P')$. 
Thus, we conclude that this open algebraic torus of $\bar{X}$
contains the generic point of $\pi^{-1}(P)$.
Henceforth, the Cox space $\bar{X}$ 
is covered, up to a closed subset of codimension two,
by open algebraic tori.
\end{proof} 

\bibliographystyle{habbvr}
\bibliography{bib}

\vspace{0.5cm}
\end{document}